\documentclass[11pt]{article}

\usepackage{amsmath,amsthm}
\usepackage{amssymb,amsfonts}
\usepackage{hyperref}
\usepackage{tikz}

\usepackage{stmaryrd}
\DeclareSymbolFont{bbold}{U}{bbold}{m}{n}
\DeclareSymbolFontAlphabet{\mathbbm}{bbold}

\usepackage{framed}							%print frame around text (per i commenti)
\usepackage{color}
\definecolor{lightgrey}{rgb}{0.95,0.95,0.95}
\definecolor{darkgrey}{rgb}{0.6,0.6,0.6}

\usepackage[mode=multiuser,status=draft,lang=english]{fixme}
\fxsetup{theme=colorsig}

\FXRegisterAuthor{M}{aM}{Matteo}
\FXRegisterAuthor{G}{aG}{Giorgio}
\FXRegisterAuthor{S}{aS}{Silvia}

\usepackage[a4paper]{geometry}
\theoremstyle{plain}
	\newtheorem*{theorem*}{Theorem}
	\newtheorem{theorem}{Theorem}[section]
	\newtheorem{proposition}[theorem]{Proposition}
	\newtheorem{lemma}[theorem]{Lemma}
	\newtheorem{corollary}[theorem]{Corollary}

\theoremstyle{definition}
	\newtheorem*{definition*}{Definition}
	\newtheorem{definition}[theorem]{Definition}

	\newtheorem{question}[theorem]{Question}
\theoremstyle{remark}

	\newtheorem{example}[theorem]{Example}

\newcommand{\id}{\ensuremath{\text{{\rm id}}}}

\newcommand{\ZFC}{\ensuremath{\text{{\sf ZFC}}}}

\newcommand{\MK}{\ensuremath{\text{{\sf MK}}}}
\DeclareMathOperator{\dom}{dom}
\DeclareMathOperator{\ran}{ran}

\DeclareMathOperator{\crit}{crit}

\DeclareMathOperator{\cof}{cof}

\DeclareMathOperator{\proj}{proj}

\DeclareMathOperator{\rank}{rank}
\DeclareMathOperator{\trcl}{trcl}
\DeclareMathOperator{\val}{val}

\DeclareMathOperator{\SkH}{SkH}
\DeclareMathOperator{\non}{non}

\newcommand{\NS}{\ensuremath{\text{{\sf NS}}}}
\newcommand{\CF}{\ensuremath{\text{{\sf CF}}}}

\newcommand{\ff}{\mathbf{F}}
\newcommand{\ii}{\mathbf{I}}

\newcommand{\BB}{\mathbb{B}}
\newcommand{\CC}{\mathbb{C}}
\newcommand{\DD}{\mathbb{D}}
\newcommand{\EE}{\mathbb{E}}

\renewcommand{\SS}{\mathbb{S}}
\newcommand{\TT}{\mathbb{T}}

\newcommand{\AAA}{\mathcal{A}}
\newcommand{\BBB}{\mathcal{B}}
\newcommand{\CCC}{\mathcal{C}}
\newcommand{\EEE}{\mathcal{E}}

\newcommand{\PPP}{\mathcal{P}}
\newcommand{\SSS}{\mathcal{S}}
\newcommand{\TTT}{\mathcal{T}}

\newcommand{\cc}{\mathfrak{c}}

\DeclareMathOperator{\Coll}{Coll}

\DeclareMathOperator{\Ult}{Ult}

\newcommand{\res}{{\upharpoonright}}
\newcommand{\cp}[1]{\left( #1 \right)}
\newcommand{\qp}[1]{\left[ #1 \right]}
\newcommand{\Qp}[1]{\left\llbracket #1 \right\rrbracket}
\newcommand{\ap}[1]{\langle #1 \rangle}
\newcommand{\bp}[1]{\left\lbrace #1 \right\rbrace}
\newcommand{\vp}[1]{\left\lvert #1 \right\rvert}
\newcommand{\0}{\mathbf 0}
\newcommand{\1}{\mathbf 1}

\newcommand{\ON}{\mathrm{\mathbf{ON}}}

\title{\textsc{Generic Large Cardinals\\and Systems of Filters}}
\author{Giorgio Audrito, Silvia Steila}
\date{}

\begin{document}
	\maketitle

	\begin{abstract}
		We introduce the notion of $\CCC$-system of filters, generalizing the standard definitions of both extenders and towers of normal ideals. This provides a framework to develop the theory of extenders and towers in a more general and concise way. In this framework we investigate the topic of definability of generic large cardinals properties.
	\end{abstract}

	\tableofcontents

	\pagebreak

	\section{Introduction} \label{sec:introduction}
		% !TEX root = GenericLargeCardinals.tex

Large cardinals have been among the most important axioms extending $\ZFC$ since the very beginning of modern set theory. On the one hand they provide a fine scale to measure the consistency strength of a rich variety of combinatorial principles, on the other hand they also solve important questions within set theory.
However, such cardinals are rarely used in common mathematics outside of set theory: for example, large parts of number theory and analysis can be formalized within $H_\cc$, and even if new subjects can push this limit beyond that point, it is uncommon for structures of inaccessible or larger size to be employed outside of set theory.

Generic large cardinal axioms try to address this point, and postulate the existence of elementary embeddings $j:V\to M$ with $M\subseteq V[G]$ a transitive class definable in a generic extension $V[G]$ of $V$. Contrary to the classical case one can consistently have generic large cardinal properties at cardinals as small as $\omega_1$. Thus, generic large cardinal axioms are fit to produce consequences on small objects, and might be able to settle questions arising in domains of mathematics other than set theory. A detailed presentation of this approach can be found in \cite{foreman:generic_embeddings}.

Due to the \emph{class} nature of the elementary embeddings involved in the definitions of large cardinals (both classical and generic), a key issue concerns the possibility to define (or derive) such embeddings from set-sized objects. The first natural candidates are ideals, although it turns out that they are not able to represent various relevant large cardinal properties. For this reason many extensions of the concept have been proposed, the most important of which are extenders (see among many \cite{claverie:ideal_extenders, kanamori:higher_infinite, koellner:large_cardinals}) and normal towers (see for example \cite{cox:viale:tower_forcing, larson:stationary_tower, viale:category_forcing, a:viale:notes_forcing}).

In this paper we introduce the notion of \emph{$\CCC$-system of filters}  (see Section~\ref{sec:system_filters}). This concept is inspired by the well-known definitions of extenders and towers of normal ideals, generalizes both of them, and provides a common framework in which the standard properties of extenders and towers used to define classical or generic large cardinals can be expressed in an elegant and concise way. Using the new framework given by $\CCC$-system of filters we easily generalize to the setting of generic large cardinals well-known results about extenders and towers, providing shorter and modular proofs of several well-known facts regarding classical and generic large cardinals. Furthermore, we are able to examine closely the relationship between extenders and towers, and investigate when they are equivalent or not, both in the standard case and in the generic one (see Section \ref{ssec:sys_derived}).

The second part of this paper investigates some natural questions regarding generic large cardinals. In particular, we first examine the difference between having a generic large cardinal property \emph{ideally} or \emph{generically}, and study when a generic $\CCC$-system of ultrafilters is able to reproduce a given large cardinal property. Then we focus on \emph{ideally} large cardinals, and study how the large cardinal properties are captured by the combinatorial structure of the $\CCC$-system of filters used to induce the embedding. In particular, we are able to characterize strongness-like properties via the notion of \emph{antichain splitting}, and closure-like properties via the notion of \emph{antichain guessing} (a generalization of the well-known concept of \emph{presaturation} for normal towers). Finally, we investigate to what extent it is possible to collapse a generic large cardinal while preserving its properties.

The remaining part of Section \ref{sec:introduction} recalls some standard terminology and some well-known results needed in the latter part of this paper. Section \ref{sec:system_filters} introduces the concept of $\CCC$-system of filters and develops their general theory. Section \ref{sec:generic_lc} addresses some issues regarding generic large cardinals, using the machinery previously developed.
\paragraph{Acknowledgements.}
We would like to thank Matteo Viale for his valuable suggestions, remarks and corrections; and for addressing us to the topic of generic large cardinals in the first place.

% % % % % % % % % % % % % % % % % % % % % % % % % % % % % % % % % % % % % % % % % % % % % % % % % % %

\subsection{Notation}

As in common set-theoretic use, $\trcl(x)$, $\rank(x)$ denote respectively the transitive closure and the rank of a given set $x$. We denote by $V_\alpha$ the sets $x$ such that $\rank(x) < \alpha$ and by $H_\kappa$ the sets $x$ such that $\vp{\trcl(x)} < \kappa$. We use $\PPP(x)$, $[x]^\kappa$, $[x]^{{<}\kappa}$ to denote the powerset, the set of subsets of size $\kappa$ and the ones of size less than $\kappa$. The notation $f : A \to B$ is improperly used to denote \emph{partial} functions in $A \times B$, ${}^AB$ to denote the collection of all such (partial) functions, and $f[A]$ to denote the pointwise image of $A$ through $f$. We denote by $\id: V \to V$ the identity map on $V$.

We say that $I \subseteq \PPP(X)$ is an \emph{ideal} on $X$ whenever it is closed under unions and subsets, and feel free to confuse an ideal with its dual filter when clear from the context. We denote the collection of $I$-positive sets by $I^+ = \PPP(X) \setminus I$.

We follow Jech's approach \cite{jech:set_theory} to forcing via boolean valued models. The letters $\BB$, $\CC$, $\DD,\ldots$ are used for set sized complete boolean algebras, and $\0$, $\1$ denote their minimal and maximal element. We  use $V^\BB$ for the boolean valued model obtained from $V$ and $\BB$, $\dot{x}$ for the elements (names) of $V^\BB$, $\check{x}$ for the canonical name for a set $x \in V$ in the boolean valued model $V^\BB$, $\Qp{\phi}_\BB$ for the truth value of the formula $\phi$.

When convenient we also use the generic filters approach to forcing. The letters $G$, $H$ will be used for generic filters over $V$, $\dot{G}_\BB$ denotes the canonical name for the generic filter for $\BB$, $\val_G(\dot{x})$ the valuation map on names by the generic filter $G$, $V[G]$ the generic extension of $V$ by $G$. Given a set $x$ and a model $M$, we denote by $M[x]$ the smallest model of $\ZFC$ including $M$ and containing $x$. Let $\phi$ be a formula. We write $V^\BB \models \phi$ to denote that $\phi$ holds in all generic extensions $V[G]$ with $G$ generic for $\BB$.

$\BB / I$ denotes the quotient of a boolean algebra $\BB$ by the ideal $I$, $\BB \ast \dot{\CC}$ denotes the two-step iteration intended as the collection of $\BB$-names for elements of $\dot{\CC}$ modulo equivalence with boolean value $\1$. References text for the results mentioned above are \cite{jech:set_theory, a:viale:notes_forcing, a:viale:semiproper_iterations}.

$\Coll(\kappa,{<}\lambda)$ is the L\'evy collapse that generically adds a surjective function from $\kappa$ to any $\gamma < \lambda$. In general we shall feel free to confuse a partial order with its boolean completion. When we believe that this convention may generate misunderstandings we shall be explicitly more careful.

Given an elementary embedding $j: V \to M$, we use $\crit(j)$ to denote the critical point of $j$. We denote by $\SkH^{M}(X)$ the Skolem Hull of the set $X$ in the structure $M$. Our reference text for large cardinals is \cite{kanamori:higher_infinite} and for generic elementary embeddings is \cite{foreman:generic_embeddings}.

% % % % % % % % % % % % % % % % % % % % % % % % % % % % % % % % % % % % % % % % % % % % % % % % % % %

\subsection{Generalized stationarity}

We now recall the main definitions and properties of generalized stationary sets, while feeling free to omit most of the proofs. Detailed references for the material covered in this section can be found in \cite{jech:set_theory}, \cite[Chp. 2]{larson:stationary_tower}, \cite{a:viale:notes_forcing}, \cite{a:viale:semiproper_iterations}.

\begin{definition}
	Let $X$ be an uncountable set. A set $C$ is a \emph{club} on $\PPP(X)$ iff there is a function $f_C: ~ X^{{<}\omega} \rightarrow X$ such that $C$ is the set of elements of $\PPP(X)$ closed under $f_C$, i.e.
	\[
	C = \bp{ Y \in \PPP(X): ~ f_C[Y]^{<\omega} \subseteq Y }
	\]
	A set $S$ is \emph{stationary} on $\PPP(X)$ iff it intersects every club on $\PPP(X)$.
\end{definition}

The reference to the support set $X$ for clubs or stationary sets may be omitted, since every set $S$ can be club or stationary only on $\bigcup S$. Examples of stationary sets are $\bp{X}$, $\PPP(X) \setminus \bp{X}$, $\qp{X}^\kappa$ for any $\kappa \leq \vp{X}$. Club sets can be thought of as containing all elementary submodels of a given structure on $X$, while stationary sets can be thought of as containing an elementary submodel of \emph{any} given first order structure on $X$ for a countable language.

We now introduce the definition and main properties of the non-stationary ideal.

\begin{definition}
	The \emph{non-stationary ideal} on $X$ is \[\NS_X = \bp{A \subset \PPP(X): ~ A \text{ not stationary}}\] and its dual filter is $\CF_X$, the \emph{club filter} on $X$.
\end{definition}

\begin{definition}
	Given an ideal $I$ on $X$, we say that $I$ (or equivalently its dual filter) is normal if for any $A \in I^+$ and for any choice function $f: ~ A \rightarrow X$ (i.e. $f(Y) \in Y$ for any $Y\in A$) there exists $x \in X$ such that $\bp{Y \in A : f(Y)=x} \in I^+$. We say that $I$ (or equivalently its dual filter) is fine if for any $x \in X$ the set $\bp{Y \subseteq X : x \notin Y }$ is in $I$.
\end{definition}

\begin{definition}
	Given a family $\bp{S_a \subseteq \PPP(X): ~ a \in X}$, the \emph{diagonal union} of the family is $\nabla_{a \in X} S_a = \bp{Y \in \PPP(X): ~ \exists a \in Y  ~ Y  \in S_a}$, and the \emph{diagonal intersection} of the family is $\Delta_{a \in X}  S_a = \bp{Y  \in \PPP(X): \forall a \in Y  ~ Y  \in S_a}$.
\end{definition}

\begin{lemma}[Fodor] \label{lem:fodor}
	$\NS_X$ is closed under diagonal union. Equivalently, $\NS_X$ is normal.
\end{lemma}

Furthermore $\NS_X$ is the smallest normal fine ideal on $X$, as shown in the following.

\begin{lemma}\label{lem:stat}
	Let $I$ be a normal and fine ideal on $X$. If $A$ is non-stationary, $A \in I$.
\end{lemma}

\begin{lemma}[Lifting and Projection] \label{lem:lifting}
	Let $X \subseteq Y$ be uncountable sets. If $S$ is stationary on $\PPP(X)$, then $S \uparrow Y = \bp{B \subseteq Y: ~ B \cap X \in S}$ is stationary. If $S$ is stationary on $\PPP(Y)$, then $S \downarrow X = \bp{B \cap X: ~ B \in S}$ is stationary.
\end{lemma}

A similar result holds for clubs: if $C$ is club on $\PPP(X)$, then $C \uparrow Y$ is a club; while if $C$ is club on $\PPP(Y)$, $C \downarrow X$ contains a club. Due to the last result the notion of non-stationary ideal can be used to define the \emph{full stationary tower} of height $\lambda$, that is the coherent collection $\ap{\NS_X : ~ X \in V_\lambda}$ for some $\lambda$.

\begin{theorem}[Ulam] \label{thm:ulam}
	Let $\kappa$ be an infinite cardinal. Then for every stationary set $S \subseteq \kappa^+$, there exists a partition of $S$ into $\kappa^+$ many disjoint stationary sets.
\end{theorem}

Stationary sets are to be intended as \emph{large} sets. Moreover, they cannot be too small even in literal sense.

\begin{lemma} \label{lem:small_ns}
	Let $S \subseteq \PPP(X) \setminus \bp{X}$ be such that $\vp{S} < \vp{X}$. Then $S$ is non-stationary.
\end{lemma}
\begin{proof}
	Let $S = S_1 \cup S_2$, $S_1 = \bp{Y \in S : ~ \vp{Y} < \vp{S}}$, $S_2 = \bp{Y \in S : ~ \vp{Y} \geq \vp{S}}$. Since $\vp{\bigcup S_1} \leq \vp{S}\cdot \vp{S}= \vp{S} < \vp{X}$, $S_1$ is non-stationary. We now prove that $S_2$ is non-stationary as well.

	Fix an enumeration $S_2 = \bp{Y_\alpha : ~ \alpha < \gamma}$ with $\gamma = \vp{S_2} < \vp{X}$. For all $\alpha < \gamma$, define recursively $x_\alpha \in X \setminus Y_\alpha$, $y_\alpha \in Y_\alpha \setminus \bp{y_\beta: ~ \beta < \alpha}$. Such $x_\alpha$ exists since $X \notin S$, and such $y_\alpha$ exists since $|Y_\alpha| \geq |S| = \gamma > \alpha$. Let $f: \qp{X}^{{<}\omega} \to X$ be such that $f(\bp{y_\alpha}) = x_\alpha$, $f(s) = x_0$ otherwise. Thus $C_f \cap S_2 = \emptyset$, hence $S_2$ is non-stationary.
\end{proof}

% % % % % % % % % % % % % % % % % % % % % % % % % % % % % % % % % % %

\subsection{Standard extenders and towers} \label{ssec:standard_systems}

We recall here the standard definitions of $\ap{\kappa,\lambda}$-extender (see e.g. \cite{koellner:large_cardinals}) and tower of height $\lambda$ (see e.g. \cite{viale:category_forcing}) in the form that is more convenient to us.

Given $a, b \in [\lambda]^{{<}\omega}$ such that\footnote{Here and in the following we assume that finite sets of ordinals are always implicitly ordered by the natural ordering on the ordinals.} $b= \bp{\alpha_0, \dots, \alpha_n} \supseteq a = \bp{\alpha_{i_0}, \dots, \alpha_{i_m}}$ and $s=\bp{s_0, \dots, s_n}$, let $\pi_{ba}(s) = \bp{s_{i_0}, \dots, s_{i_m}}$.

\begin{definition}\label{def:standardext}
	$\EE = \bp{F_a : ~ a \in \qp{\lambda}^{{<}\omega}}$ is a standard $\ap{\kappa,\lambda}$-extender with supports $\ap{\kappa_a : a \in [\lambda]^{{<}\omega}}$ iff the following holds.
	\begin{enumerate}
		\item (Filter property) For all $a \in [\lambda]^{{<}\omega}$, $F_a$ is a ${<}\kappa$-complete filter on $[\kappa_a]^{\vp{a}}$ and $\kappa_a$ is the least $\xi$ such that $[\xi]^{\vp{a}} \in F_a$;
		\item (Compatibility) if $a \subseteq b \in [\lambda]^{{<}\omega}$ then
		\begin{enumerate}
			\item $\kappa_a \leq \kappa_b$;
			\item if $\max(a)=\max(b)$, then $\kappa_a =\kappa_b$;
			\item $A \in F_a$  iff $ {\pi}^{-1}_{ba}[A] \in F_b$;
		\end{enumerate}
		\item (Uniformity) $\kappa_{\bp{\kappa}}=\kappa$;
		\item (Normality) Assume that $a \in[\lambda]^{{<}\omega}$, $A \in I_a^+$ where $I_a$ is the dual of $F_a$, $u:A \to \kappa_a$, $i \in |a|$ are such that $u(s) \in s_i$ for all $s \in A$. Then there exist $\beta \in a_i$, $b \supseteq a \cup \bp{\beta}$ and  $B \leq_{\EE} A$ (i.e. such that $\pi_{ba}^{-1}[A]\supseteq B$) with $B \in {I}_b^+$ such that for all $s \in B$, $u(\pi_{ba}(s)) = s_j$, where $b_j=\beta$.
	\end{enumerate}
\end{definition}

\begin{definition}\label{def:standardtow}
	$\TT = \bp{F_a : ~ a \in  V_\lambda}$ is a standard tower of height $\lambda$ iff the following holds.
	\begin{enumerate}
		\item (Filter property) For all $a \in V_\lambda$, $F_a$ is a non trivial filter on $\PPP(a)$;
		\item (Compatibility) For all $a \subseteq b$, $A \in F_a$ iff $A \uparrow b = \bp{X \subseteq b : ~ X \cap a \in A} \in F_b$;
		\item (Fineness)  For all $a \in V_\lambda$ and $x \in a$ we have $\bp{X \subseteq a : ~ x \in X} \in F_a$;
		\item (Normality) Given $A \in {I}_a^+$, $u: A \to V$ such that $u(X) \in X$ for any $X \in A$, there exist $b\supseteq a$, $B \in {I}^+_b$ with $B\leq_\TT A$ (i.e. such that $A\uparrow b\supseteq B$),  and a fixed $y$ such that $u(X\cap a) =y$ for all $X \in B$.
	\end{enumerate}
\end{definition}

	\section{Systems of filters} \label{sec:system_filters}
		% !TEX root = GenericLargeCardinals.tex

In this section we present the definition and main properties of $\CCC$-systems of filters. This notion has both classical extenders \cite{kanamori:higher_infinite, koellner:large_cardinals}, ideal extenders (recently introduced by Claverie in \cite{claverie:ideal_extenders}) and towers \cite{cox:viale:tower_forcing, larson:stationary_tower, viale:category_forcing} as special cases, and it is able to generalize and subsume most of the standard results about extenders and towers.

Throughout this section let $V$ denote a transitive model of $\ZFC$.

\begin{definition}
	We say that a set $\CCC\in V$ is a \emph{directed set of domains} iff the following holds:
	\begin{enumerate}
		\item \emph{(Ideal property)}  $\CCC$ is closed under subsets and unions;
		\item \emph{(Transitivity)} $\bigcup \CCC$ is transitive, i.e. for every $y \in x \in a \in \CCC$ we have $y \in \bigcup \CCC$ (or, equivalently in presence of the ideal property, $\bp{y} \in \CCC$).
	\end{enumerate}
	We say that $\CCC$ has length $\lambda$ iff $\rank(\CCC) = \lambda$, and that $\CCC$ is ${<}\gamma$-directed iff it is closed under unions of size ${<}\gamma$ in $V$.
\end{definition}

\begin{quote}
\begin{example}
	In the case of extenders, $\CCC$ will be $\qp{\lambda}^{{<}\omega}$, while for towers it will be $V_\lambda$. The first is absolute between transitive models of $\ZFC$, while the latter is ${<}\lambda$-directed whenever $\lambda$ is regular. These two different properties entail most of the differences in behaviour of these two objects.
\end{example}
\end{quote}

\begin{definition}
	Let $\CCC\in V$ be a directed set of domains. Given a domain $a \in \CCC$, we define $O_a$ as the set of functions
	\[
	O_a = \bp{\pi_M \res (a \cap M) : ~ M \subseteq \trcl(a),\, M\in V \text{ extensional}}
	\]
	where $\pi_M$ is the Mostowski collapse map of $M$. If $a \subseteq b$, we define the \emph{standard projection} $\pi_{ba} : O_b \to O_a$ by $\pi_{ba}(f) = f \res a$. 
\end{definition}

We shall sometimes denote $\pi_{ba}$ by $\pi_a$ and $\pi_{ba}^{-1}$ by $\pi_b^{-1}$ when convenient. Notice that every $f \in O_b$ is $\in$-preserving, and that $\pi_{ba}(f) = f \res a \in O_a$ for any $a \subseteq b$, so that $\pi_{ba}$ is everywhere defined. From now on we shall focus on filters on the boolean algebra $\PPP^V(O_a)$ for $a \in \CCC$ and $\CCC\in V$ a directed set of domains.

\begin{quote}
\begin{example}
	In the case of extenders, any $f \in O_a$ will be an increasing function from the sequence $a \in \qp{\lambda}^{{<}\omega}$ to smaller ordinals. $O_a$ can be put in correspondence with the domain $\kappa_a^{\vp{a}}$ of a standard extender via the mapping $f \mapsto \ran(f)$, $\pi_{ba}$ will correspond in the new setting to the usual notion of projection for extenders.

	In the case of towers, any $f \in O_a$ with $a$ transitive will be the collapsing map of a $M \subseteq a$. In this case $O_a$ can be put in correspondence with the classical domain $\PPP^V(a)$ via the mapping $f \mapsto \dom(f)$, and $\pi_{ba}$ will correspond to the usual notion of projection for towers.
	
	A complete proof of the above mentioned equivalences can be found in Section \ref{ssec:ext_tow_systems}. 
\end{example}
\end{quote}

\begin{definition}
	Define $x \trianglelefteq y$ as $x \in y \vee x = y$. We say that $u : O_a \to V$ is \emph{regressive} on $A \subseteq O_a$ iff for all $f \in A$, $u(f) \trianglelefteq f(x_f)$ for some $x_f \in \dom(f)$. We say that $u$ is \emph{guessed} on $B \subseteq O_b$, $b \supseteq a$ iff there is a fixed $y \in b$ such that for all $f \in B$, $u(\pi_{ba}(f)) = f(y)$.
\end{definition}

\begin{definition} \label{def:csystem}
	Let $V \subseteq W$ be transitive models of $\ZFC$ and $\CCC \in V$ be a directed set of domains. We say that $\SS = \bp{F_a : ~ a \in \CCC}\in W$ is a $\CCC$-system of $V$-filters, and we equivalently denote $\SS$ also by $\bp{I_a : ~ a \in \CCC}$ where $I_a$ is the dual ideal of $F_a$, iff the following holds:
	\begin{enumerate}
		\item \emph{(Filter property)} for all $a \in \CCC$, $F_a$ is a non-trivial filter on the boolean algebra $\PPP^V(O_a)$;
		\item \emph{(Fineness)} for all $a \in \CCC$ and $x \in a$, $\bp{f \in O_a : ~ x \in \dom(f)} \in F_a$;
		\item \emph{(Compatibility)} for all $a \subseteq b$ in $\CCC$ and $A \subseteq O_a$, $A \in F_a \iff \pi_{ba}^{-1}[A] \in F_b$;
		\item \emph{(Normality)} every function $u : A\to V$ in $V$ that is regressive on a set $A \in I_a^+$ for some $a \in \CCC$ is guessed on a set $B \in I_b^+$ for some $b \in \CCC$ such that $B \subseteq \pi^{-1}_{ba}[A]$;
	\end{enumerate}

	We say that $\SSS$ is a $\CCC$-system of $V$-ultrafilters if in addition:
	\begin{enumerate}
		\item[5.] \emph{(Ultrafilter)} for all $a \in \CCC$, $F_a$ is an ultrafilter on $\PPP^V(O_a)$.
	\end{enumerate}
\end{definition}

We shall feel free to drop the reference to $V$ when clear from the context, hence denote the $\CCC$-systems of $V$-filters as $\CCC$-systems of filters. When we believe that this convention may generate misunderstandings we shall be explicitly more careful. To clearly distinguish $\CCC$-systems of filters from $\CCC$-systems of ultrafilters, in the following we shall use $\SS$, $\EE$, $\TT$ for the first and $\SSS$, $\EEE$, $\TTT$ for the latter.

\begin{definition}
	Let $\SS$ be a $\CCC$-system of filters, $a$ be in $\CCC$. We say that $\kappa_a$ is the \emph{support} of $a$ iff it is the minimum $\alpha$ such that $O_a \cap {}^a V_\alpha \in F_a$. We say that $\SS$ is a $\ap{\kappa, \lambda}$-system of filters if and only if:
	\begin{itemize}
		\item it has length $\lambda$ and $\kappa \subseteq \bigcup \CCC$,
		\item $F_{\bp{\gamma}}$ is principal generated by $\id \res \bp{\gamma}$ whenever $\gamma < \kappa$,
		\item $\kappa_a \leq \kappa$ whenever $a \in V_{\kappa+2}$.
	\end{itemize}
\end{definition}

Notice that $\kappa_a \leq \rank(a)$, and $\kappa_a = \rank(a)$ when $F_a$ is principal as in the above definition. In particular, $\kappa_{\bp{\gamma}} = \gamma+1$ in this case. The definition of $\CCC$-system of filters entails several other properties commonly required for coherent systems of filters.

\begin{proposition} \label{prop:largeness}
	Let $\SS$ be a $\CCC$-system of filters. Then $dF_a = \bp{\dom[A] : ~ A \in F_a}$ is a normal and fine filter on $\PPP(a)$, for any $a$ in $\CCC$ infinite. In particular if $a$ is uncountable, $\dom[A]$ is stationary for all $A \in F_a$.
\end{proposition}
\begin{proof}
	Filter property and fineness follow directly from restricting the corresponding points in Definition \ref{def:csystem} to $\dom[O_a]$. We now focus on normality. Let $u : D \to a$ where $D = \dom[A]$ be such that $u(X) \in X$ for all $X \in D$ (i.e. $X = \dom(f)$ for some $f$ in $A$). Then we can define $v: A \to V$ as $v(f) = f(u(\dom(f)))$. Let $B \in I_b^+$, $y \in b$ be such that $v(\pi_{ba}(f)) = f(u(\dom(\pi_{ba}(f)))) = f(y)$ for all $f \in B$ by normality. Since every $f \in B$ is injective, $u(\dom(\pi_{ba}(f))) = y$ for all $f \in B$ hence $u$ is constant on $\dom[B] \in dI_a^+$. By Lemma \ref{lem:stat} if $a$ is uncountable we conclude that $\dom[A]$ is stationary for any $A \in F_a$.
\end{proof}

\begin{proposition} \label{prop:support_monotonicity}
	Let $\SS$ be a $\CCC$-system of filters, $a, b$ be in $\CCC$. Then $\rank(a) \leq \rank(b)$ implies that $\kappa_a \leq \kappa_b$, i.e. the supports depend (monotonically) only on the ranks of the domains.
\end{proposition}

\begin{proposition} \label{prop:completeness}
	Let $\SS$ be a $\ap{\kappa, \lambda}$-system of filters, $a$ be in $\CCC$. Then $F_a$ is ${<}\kappa$-complete.
\end{proposition}

We defer the proof of the last two propositions to Section \ref{ssec:ultrapowers} (just before Proposition \ref{prop:critical_point}) for our convenience. We are now ready to introduce the main practical examples of $\CCC$-system of filters.

\begin{definition} \label{def:tow_ext}
	Let $V\subseteq W$ be transitive models of $\ZFC$.
	
	$\EE\in W$ is an \emph{ideal extender} on $V$ iff it is a $[\lambda]^{{<}\omega}$-system of filters on $V$ for some $\lambda$. $\EEE\in W$ is an \emph{extender} on $V$ iff it is a $[\lambda]^{{<}\omega}$-system of ultrafilters on $V$.

	$\EE\in W$ is an \emph{ideal $\gamma$-extender} on $V$ iff it is a $([\lambda]^{{<}\gamma})^V$-system of filters for some $\lambda$. $\EEE$ is a \emph{$\gamma$-extender} on $V$ iff it is a $([\lambda]^{{<}\gamma})^V$-system of ultrafilters on $V$.

	$\TT \in W$ is an \emph{ideal tower} iff it is a $V_\lambda$-system of filters for some $\lambda$. $\TTT \in W$ is a \emph{tower} iff it is a $V_\lambda$-system of ultrafilters.
\end{definition}

The above definitions of extender and tower can be proven equivalent to the classical ones defined in Section \ref{ssec:standard_systems} (see also \cite{koellner:large_cardinals, viale:category_forcing}) via the mappings $\ran_a : O_a \to \qp{\kappa_a}^{\vp{a}}$, $f \mapsto \ran(f)$ (for extenders) and $\dom_a : O_a \to \PPP(a)$, $f \mapsto \dom(f)$ (for towers). Furthermore, $\ap{\kappa_a : ~ a \in \CCC}$ correspond to the supports of long extenders as defined in \cite{koellner:large_cardinals}. A detailed account of this correspondence is given in Section \ref{ssec:ext_tow_systems}.

Given a $\CCC$-system of $V$-filters $\SS$, we can define a preorder $\leq_\SS$ on the collection $\SS^+ = \bp{A : ~ \exists a \in \CCC ~ A \in I_a^+}$ as in the following.

\begin{definition}
	Given $A \in I_a^+$, $B \in I_b^+$ we say that $A \leq_{\SS} B$ iff $\pi^{-1}_{ca}[A] \leq_{I_c} \pi^{-1}_{cb}[B]$ where $c = a \cup b$, and $A =_{\SS} B$ iff  $A \leq_{\SS} B$ and  $B \leq_{\SS} A$.
\end{definition}

Consider the quotient $\SS^+/=_\SS$. With an abuse of notation for $p, q \in \SS^+/=_{\SS}$, we let $p \leq_{\SS} q$  iff $A \leq_\SS B$ for any (some) $A \in p$, $B \in q$. The partial order $\ap{\SS^+/=_\SS, \leq_\SS}$ is a boolean algebra which is the limit of a directed system of boolean algebras, and can be used as a forcing notion in order to turn $\SS$ into a system of ultrafilters. This process will be described in Section \ref{ssec:filters_generic}.

\begin{proposition}
	Let $\CCC$ be a ${<}\gamma$-directed set of domains, $\SS$ be a $\CCC$-system of filters. Then $\ap{\SS^+/=_\SS, \leq_\SS}$ forms a ${<}\gamma$-closed boolean algebra.
\end{proposition}
\begin{proof}
	Let $\AAA = \ap{A_\alpha : ~ \alpha < \mu} \subseteq \SS^+$ be such that $\mu < \gamma$ with $A_\alpha\in I_{a_\alpha}^+$. Since $\CCC$ is ${<}\gamma$-directed, there is a domain $a \in \CCC$ with $\vp{a} \geq \mu$ such that $a_\alpha \subseteq a$ for all $\alpha < \mu$. Fix $\ap{x_\alpha : ~ \alpha < \mu}$ a (partial) enumeration of $a$, and define 
	\[
	B = \bp{f \in O_a : ~ \forall \alpha < \mu ~ x_\alpha \in \dom(f) \Rightarrow f \in \pi^{-1}_a[A_\alpha]}.
	\]

	First, $B <_\SS A_\alpha$ for all $\alpha < \mu$ by fineness, since $\bp{f \in B : ~ x_\alpha \in \dom(f)} \subseteq \pi^{-1}_{a}[A_\alpha]$. Suppose now by contradiction that for some $c\supseteq a$, $C \in I_c^+$ is such that $C \leq_\SS A_\alpha$ for all $\alpha < \mu$ and $C \cap \pi_c^{-1}[B] = \emptyset$. Then for any $f \in C$ we can find an $\alpha_f < \mu$ such that $x_{\alpha_f} \in \dom(f)$ and $f \notin \pi^{-1}_c[A_{\alpha_f}]$. Define
	\[
	\begin{array}{rrcl}
		u: & C & \longrightarrow & V \\
		   & f & \longmapsto & f(x_{\alpha_f})
	\end{array}
	\]
	By normality we can find a single $\bar{\alpha}$ and a $d\supseteq c\cup\bp{\bar{\alpha}}$ such that 
	\[
	D = \bp{f \in \pi^{-1}_d[C]: ~ u(\pi_d(f)) = f(x_{\bar{\alpha}})} \in I_d^+.
	\] 
	Thus $D \cap A_{\bar{\alpha}} = \emptyset$ and $D \leq_\SS C \leq_\SS A_{\bar{\alpha}}$, a contradiction.
\end{proof}

% % % % % % % % % % % % % % % % % % % % % % % % % % % % % % % % % % % % % % % % % % % % % % % % % % %

\subsection{Standard extenders and towers as $\CCC$-systems of filters} \label{ssec:ext_tow_systems}

\paragraph{Extenders.} We now compare the definition of $\ap{\kappa,\lambda}$-extender just introduced (Definition \ref{def:tow_ext}) with the definition of standard $\ap{\kappa,\lambda}$-extender (Definition \ref{def:standardext}). 

Let $\EE$ be a $\ap{\kappa,\lambda}$-extender with supports $\ap{\kappa_a : a \in [\lambda]^{{<}\omega}}$ according to Definition \ref{def:tow_ext}. Notice that given any $a \in \qp{\lambda}^{{<}\omega}$, the collection 
\[
O'_a = \bp{f \in O_a : ~ \dom(f) = a \wedge \ran(f) \subseteq \kappa_a}
\] 
is in $F_a$ by fineness and definition of $\kappa_a$. Consider the injective map $\ran_a : O'_a \to \qp{\kappa_a}^{\vp{a}}$, which maps $F_a$ into a corresponding filter $F'_a$ on $\qp{\kappa_a}^{\vp{a}}$ that is the closure under supersets of $\bp{\ran_a[A \cap O'_a] : ~ A \in F_a}$. Notice that many sequences $s \in \qp{\kappa_a}^{\vp{a}}$ cannot be obtained as the range of Mostowski collapse maps, e.g. $s = \bp{\beta,\beta+2}$ whenever $a$ is of the kind $\bp{\alpha, \alpha+1}$.

Let us denote with $\pi'_{ba}$ the projection map from $F'_b$ to $F'_a$ in the standard case. Notice that for any $a \subseteq b \in  [\lambda]^{{<}\omega}$ and $f\in O_b$, $\ran_a(\pi_{ba}(f)) = \pi'_{ba}(\ran_b(f))$.

Define $\EE' = \bp{F'_a: a \in  \qp{\lambda}^{{<}\omega}}$. We claim that $\EE'$ is a $\ap{\kappa,\lambda}$-extender with respect to the standard definition whenever $\EE$ is a $\ap{\kappa,\lambda}$-extender.

\begin{proposition} \label{prop:std_ext_1}
	If $\EE$ is a $\ap{\kappa, \lambda}$-extender then $\EE'$ is a standard $\ap{\kappa,\lambda}$-extender.
\end{proposition}
\begin{proof}
	\begin{enumerate}
	\item (Filter property) It follows since $F'_a$ is an injective image of $F_a \res O'_a$.
	\item (Compatibility)
	\begin{enumerate}
		\item[(a-b)] Follow by Proposition \ref{prop:support_monotonicity}, since $\rank(a)$ depends only on $\max(a)$.
		\item[(c)] By compatibility of $\EE$, it follows: $A'= \ran_a[A] \in F'_a$ iff  $A \in F_a$ iff $\pi^{-1}_{ba}[A] \in F_b$ iff $\ran_b[\pi^{-1}_{ba}[A]] = {\pi'}^{-1}_{ba}[A] \in F_b'$.
	\end{enumerate}
	\item (Uniformity) By definition of $\ap{\kappa, \lambda}$-system of filters.
	\item (Normality) Given $a \in[\lambda]^{{<}\omega}$, $A'=\ran_a[A] \in {I'}_a^+$, $u:A' \to \kappa_a$, $i < |a|$ such that $u(s) \in s(i)$ for all $s \in A'$, let $\alpha= a_i$. Define 
	\[
	\begin{array}{rrcl}
		v: & A & \longrightarrow  & V \\
		   & f & \longmapsto & u(\ran_a(f)).
	\end{array}
	\] 
	Since $s_i = f(\alpha)$, $v$ is regressive. By normality of $\EE$, there exist $\beta$ and $B \subseteq \pi^{-1}_{ba}[A]$, where $\beta \in b \supseteq a$ such that for all $f \in B$, $v(\pi_{ba}(f))= f(\beta)$. Since $\pi_{ba}(f)\in A$, $f(\beta) = v(\pi_{ba}(f)) \in f(\alpha)$. Since $f$ is ${\in}$-preserving, $\beta \in \alpha$ and $B' = \ran_b[B]$ witnesses normality of $\EE'$. \qedhere
	\end{enumerate}
\end{proof}

On the other hand, given a standard $\ap{\kappa, \lambda}$-extender $\EE'$ we can define a collection of corresponding filters $F_a$ on $O_a'$ for any $a \in [\lambda]^{{<}\omega}$. This can be achieved since $\ran_a[O'_a] \in F'_a$ for any $a \in [\lambda]^{{<}\omega}$ and standard $\ap{\kappa, \lambda}$-extender $\EE'$.\footnote{This fact can be proved directly using the corresponding version of Proposition \ref{prop:induced_system} (i.e. $A \in F'_a$ iff $a \in j(A)$) and \L o\'s Theorem \ref{thm:los} for standard extenders.} Let $\EE$ consists of the closure of $F_a$ under supersets in $O_a$, for any $a \in [\lambda]^{{<}\omega}$. Then we can show the following.

\begin{proposition} \label{prop:std_ext_2}
	If $\EE'$ is a standard $\ap{\kappa, \lambda}$-extender then  $\EE$ is a $\ap{\kappa,\lambda}$-extender.
\end{proposition}
\begin{proof}
	\begin{enumerate}
	\item (Filter property) Follows directly from the filter property of $\EE'$.
	\item (Compatibility) By compatibility of $\EE'$ and unfolding definitions, $A \in F_a$ iff $A \cap O'_a \in F_a$ iff  $\ran_a[A \cap O'_a] \in F'_a$ iff $ \ran_b[\pi^{-1}_{ba}[A \cap O'_a]] \in F_b'$ iff $\pi^{-1}_{ba}[A \cap O'_a] \in F_b$ iff $\pi^{-1}_{ba}[A] \in F_b$.
	\item (Fineness) For any $x \in a$, $\bp{f \in O_a : ~ x \in \dom(f)} \supseteq O'_a \in F_a$.
	\item (Normality) Assume that $A \in I_a^+$ and that $u : A \to V$ is regressive. By definition of regressive function we have that $A= A_0 \cup A_1$, where  
	\[
	A_0 =\bp{f \in A : \exists x \in \dom_a(f)(u(f) = f(x))};
	\]
	\[
	A_1=\bp{f \in A : \exists x \in \dom_a(f)(u(f) \in f(x))}.
	\]
	We have two cases.
	\begin{itemize}
		\item If $A_0  \in I_a^+$ there exists a fixed $x \in a$ such that $B = \bp{f \in A : u(f) = f(x)}$ is in $I_a^+$ since $a$ is finite. Hence $B$ and $x$ witness normality for $\EE$. 
		\item Otherwise, if $A_0 \in I_a$, we have that $A_1 \in I_a^+$. Since $a$ is finite, there exists $x = a_i \in a$ such that $A^*= \bp{f \in A : u(f) \in f(x)} \in I_a^+$. Let $A'= \ran_a[A^* \cap O'_a]$. Let $v: A'  \to  \kappa_a$ be such that $v(s) = u(f)$ for any $s = \ran_a(f)$, so that $v(s) \in s_i$. By normality of $\EE'$ there exist $\beta \in a_i$ and  $B' <_{\EE'} A'$ with $B' = \ran_b[B] \in {I'}_b^+$ for $b = a \cup \bp{\beta}$ such that for all $s \in B'$, $v(\pi'_{ba}(s)) = s_j$, where $b_j = \beta$. Hence for any $f \in B$, $u(\pi_{ba}(f))=f(\beta)$. \qedhere
	\end{itemize}
	\end{enumerate}
\end{proof}

\paragraph{Towers.} We now compare the definition of tower just introduced (Def. \ref{def:tow_ext}) with the definition of standard tower (Def. \ref{def:standardtow}). Let $\TT$ be a tower of length $\lambda$. Notice that the whole tower can be induced from the filters $F_a$ where $a$ is a transitive set. Furthermore, whenever $a$ is transitive the map $\dom_a : O_a \to \PPP(a)$ is a bijection. In fact, any $f \in O_a$ with $\dom(f) = X$ has to be $f = \pi_X$. Thus we can map any $F_a$ with $a$ transitive into an isomorphic filter $F'_a = \bp{\dom_a[A] : ~ A \in F_a}$ on $\PPP(a)$. Define $\TT' = \bp{F'_a: ~ a \in  V_\lambda}$, then we can prove the following.

\begin{proposition}
	If $\TT$ is a tower of length $\lambda$ then $\TT'$ is a standard tower of length $\lambda$.
\end{proposition}
\begin{proof}
	The proof follows the same strategies used in Propositions \ref{prop:std_ext_1}, \ref{prop:std_ext_2} and is left to the reader.
\end{proof}

% % % % % % % % % % % % % % % % % % % % % % % % % % % % % % % % % % % % % % % % % % % % % % % % % % %

\subsection{Systems of filters in $V$ and generic systems of ultrafilters} \label{ssec:filters_generic}

In this section we shall focus on ideal extenders and ideal towers in $V$, and their relationship with the corresponding generic systems of ultrafilters. This relation will expand from the following bidirectional procedure, mapping a $V$-ultrafilter in a generic extension with an ideal in $V$ and viceversa. Full references on this procedure can be found in \cite{foreman:generic_embeddings}.

\begin{definition}
	Let $\dot{F}$ be a $\BB$-name for an ultrafilter on $\PPP^V(X)$. Let $\ii(\dot{F})\in V$ be the ideal on $\PPP^V(X)$ defined by:
	\[
	\ii(\dot{F}) = \bp{Y \subset X: ~ \Qp{\check{Y} \in \dot{F}}_\BB = \0}
	\]

	Conversely, let $I$ be an ideal in $V$ on $\PPP(X)$ and consider the poset $\CC = \PPP(X) / I$. Let $\dot{\ff}(I)$ be the $\CC$-name for the $V$-generic ultrafilter for $\PPP^V(X)$ defined by:
	\[
	\dot{\ff}(I) = \bp{\ap{\check{Y}, \qp{Y}_I} : ~ Y \subseteq X}
	\]
\end{definition}

Notice that $\ii(\dot{\ff}(I)) = I$, while the $\BB$-name $\dot{F}$ and the $\CC$-name $\dot{\ff}(\ii(\dot{F}))$ might be totally unrelated (since $\CC = \PPP(X)/\ii(\dot{F})$ does not necessarily embeds completely into $\BB$). We refer to Theorem \ref{thm:failure} and subsequent corollary for an example of this behavior.

\begin{definition}
	Let $\dot{F}$ be a $\BB$-name for an ultrafilter on $\PPP^V(X)$. Set $\CC = \PPP(X) / \ii(\dot{F})$. The \emph{immersion} map $i_{\dot{F}}$ is defined as follows:
	\[
	\begin{array}{rrcl}
		i_{\dot{F}}: &\CC & \longrightarrow &\BB \\
		&[A]_{\ii(\dot{F})} & \longmapsto &\Qp{\check{A} \in \dot{F}}_\BB
	\end{array}
	\]
\end{definition}

\begin{proposition}
	Let $\dot{F}$, $i_{\dot{F}}$ be as in the previous definition. Then $i_{\dot{F}}$ is a (not necessarily complete) morphism of boolean algebras.
\end{proposition}
\begin{proof}
	By definition of $\ii(\dot{F})$, the morphism is well-defined. Since $\dot{F}$ is a $\BB$-name for a filter $i_{\dot{F}}$ preserves the order of boolean algebras, and since $\dot{F}$ satisfies the ultrafilter property it also preserves complementation.
\end{proof}

The above can be immediately extended to systems of filters, by means of the following.

\begin{definition} \label{def:ideal_generic}
	Let $\dot{\SSS} = \ap{\dot{F}_a : ~ a \in \CCC}$ be a $\BB$-name for a $\CCC$-system of ultrafilters. Then $\ii(\dot{\SSS}) = \ap{I_a = \ii(\dot{F}_a) : ~ a \in \CCC}$ is the corresponding system of filters in $V$.
	Conversely, let $\SS = \ap{I_a : ~ a \in \CCC}$ be a $\CCC$-system of filters in $V$. Then $\dot{\ff}(\SS) = \ap{\dot{F}_a = \dot{\ff}(I_a) : ~ a \in \CCC}$ is the corresponding $\SS$-name for a system of ultrafilters.
\end{definition}

\begin{proposition} \label{prop:ideal_generic}
	Let $\dot{\SSS}$ be a $\BB$-name for a $\CCC$-system of ultrafilters.  Then $\ii(\dot{\SSS})$ is a $\CCC$-system of filters in $V$. Conversely, let $\SS$ be a $\CCC$-system of filters in $V$. Then $\dot{\ff}(\SS)$ is the canonical $\ap{\SS^+, <_\SS}$-name for the $V$-generic filter on $\SS$ and defines a $\CCC$-system of $V$-ultrafilters.
\end{proposition}
\begin{proof} $\ii(\dot{\SSS})$ and $\dot{\ff}(\SS)$ satisfy the following properties.
	\begin{enumerate}
		\item \emph{(Filter and ultrafilter property)}, \emph{(Fineness)}, and \emph{(Compatibility)} are left to the reader.

		\item \emph{(Normality)} Let $u : O_a \to V$ in $V$ be regressive on $A$. 
		
		$\ii(\dot{\SSS})$: Suppose that $A \in \ii(\dot{\SSS})^+$ (i.e. $\Qp{\check{A} \in \dot{F}_a} = p > \0$). Then $p \Vdash \check{A} \in \dot{F}_a$ implies that
		\[
		p \Vdash \exists X <_{\dot{\SSS}} \check{A} ~ X \in \dot{\SSS}^+ ~ \exists y \in \check{b} ~ \forall f \in X ~ \check{u}(\pi_{ba}(f)) =  f(y)
		\]
		Thus by the forcing theorem there is a $q < p$, $q > \0$ and fixed $B \subseteq \pi^{-1}_{ba}[A]$, $y \in b$ such that $q$ forces the above formula with the quantified $X$ replaced by $B$. Then $\Qp{\check{B} \in \dot{\SSS}^+} \geq q > \0 \Rightarrow B \in \ii(\dot{\SSS})^+$, and $\forall f \in B ~ u(\pi_{ba}(f)) = f(y)$ holds true in $V$.

		$\dot{\ff}(\SS)$: Consider the system $\dot{\ff}(\SS)$. Suppose that $A\in I_a^+$.  Given any $C \leq_\SS A$ in $I_c^+$, we can find $B \leq_\SS C$ in $\SS^+$ witnessing the normality of $\SS$ for the regressive map on $C$ defined by $h\mapsto u(h\restriction a)$. We conclude that there are densely many $B$ below $A$ such that $\exists y \in b ~ \forall f \in B ~ u(\pi_{ba}(f)) = f(y)$, hence
		\[
		\Qp{\exists B <_{\dot{\ff}(\SS)} \check{A} ~ B \in \dot{\ff}(\SS)^+ ~ \exists y \in \check{b} ~ \forall f \in B ~ \check{u}(\pi_{ba}(f)) =  f(y)} \geq [A]_\SS = \Qp{\check{A} \in \dot{\ff}(\SS)^+}
		\]
	\end{enumerate}
\end{proof}

As already noticed for single filters, the maps $\ii$ and $\dot{\ff}$ are not inverse of each other and $\dot{\ff}(\ii(\dot{\SSS}))$ might differ from $\dot{\SSS}$.

% % % % % % % % % % % % % % % % % % % % % % % % % % % % % % % % % % % % % % % % % % % % % % % % % % %

\subsection{Embedding derived from a system of ultrafilters} \label{ssec:ultrapowers}

We now introduce a notion of \emph{ultrapower} induced by a $\CCC$-system of $V$-ultrafilters $\SSS$. Notice that the results of the last section allows to translate any result about $\CCC$-systems of $V$-ultrafilters to a result on $\CCC$-systems of filters in $V$, by simply considering the $\CCC$-system of $V$-ultrafilters $\dot{\ff}(\SS)$.

\begin{definition}
	Let $V\subseteq W$ be transitive models of $\ZFC$ and $\SSS\in W$ be a $\CCC$-system of $V$-ultrafilters. Let
	\[
	U_\SSS = \bp{u: O_a \to V ~ : ~ a \in \CCC,~ u\in V}.
	\]
	Define the relations
	\[
	u =_\SSS v ~ \iff ~ \bp{f \in O_c : ~ u(\pi_{ca}(f)) = v(\pi_{cb}(f))} \in F_c
	\]
	\[
	u \in_\SSS v ~ \iff ~ \bp{f \in O_c : ~ u(\pi_{ca}(f)) \in v(\pi_{cb}(f))} \in F_c
	\]
	where $O_a = \dom(u)$, $O_b = \dom(v)$, $c = a \cup b$. The ultrapower of $V$ by $\SSS$ is $\Ult(V, \SSS) = \ap{U_\SSS / =_\SSS, ~ \in_\SSS}$.
\end{definition}

We leave to the reader to check that the latter definition is well-posed. From now on, we identify the well-founded part of the ultrapower with its Mostowski collapse.

\begin{theorem}[\L o\'s] \label{thm:los}
	Let $\phi(x_1, \ldots, x_n)$ be a formula and let $u_1, \ldots, u_n \in U_\SSS$. Then $\Ult(V, \SSS) \models \phi([u_1]_\SSS, \ldots, [u_n]_\SSS)$ if and only if
	\[
	\bp{f \in O_b : ~ \phi(u_1(\pi_{ba_1}(f)), \ldots, u_n(\pi_{ba_n}(f)))} = A \in F_b
	\]
	where $O_{a_i} = \dom(u_i)$ for $i=1\ldots n$, $b = \bigcup a_i$.
\end{theorem}
\begin{proof}
	Left to the reader, it is a straightforward generalization of \L o\'s theorem for directed systems of ultrapowers to the current setting.
\end{proof}

As in common model-theoretic use, define $j_\SSS : V \to \Ult(V, \SSS)$ by $j_\SSS(x) = [c_x]_\SSS$ where $c_x : O_\emptyset \to \bp{x}$. From the last theorem it follows that the map $j_\SSS$ is elementary. Notice that the proof of the last theorem does not use neither \emph{fineness} nor \emph{normality} of the system of ultrafilters. However these properties allows us to study the elements of the ultrapower by means of the following proposition.

\begin{proposition} \label{prop:rappcanonici}
	Let $\SSS$ be a $\CCC$-system of ultrafilters, $j : V \to M = \Ult(V, \SSS)$ be the derived embedding. Then,
	\begin{enumerate}
		\item $[c_x]_\SSS = j(x)$ for any $x \in V$;
		\item $[\proj_x]_\SSS = x$ for any $x \in \bigcup \CCC$, where
		\[
		\begin{array}{rrcl}
			   \proj_x: & O_{\bp{x}} & \longrightarrow  & V \\
			   & f & \longmapsto & f(x)
		\end{array}
		\]
		\item $[\ran_a]_\SSS = a$ for any $a \in \CCC$;
		\item $[\dom_a]_\SSS = j[a]$ for any $a \in \CCC$;
		\item $[\id_a]_\SSS = \cp{j \res a}^{-1}$ for any $a \in \CCC$.
	\end{enumerate}
\end{proposition}
\begin{proof}
	\begin{enumerate}
		\item Follows from the definition of $j$.

		\item By induction on $\rank(x)$. Fix $x \in \bigcup \CCC$. If $y \in x$, then $y = [\proj_y]_\SSS \in [\proj_x]_\SSS$ since all $f$ in $O_{\bp{x,y}}$ are ${\in}$-preserving. Conversely, assume that $x \in a$ and $[u:O_a \to V]_\SSS \in [\proj_{x}]_\SSS$. By \L o\'s's Theorem, $u$ is regressive on $A = \bp{f \in O_a: ~ u(f) \in f(x)} \in F_a$ thus there exist $y \in b \supseteq a$ and $B \subseteq \pi_b^{-1}[A]$ in $F_b$ such that $u (\pi_a(f)) = f(y)$ for all $f \in B$. Since $f(y) = u (\pi_a(f)) \in f(x)$ and any $f \in B$ is ${\in}$-preserving, it follows that $y \in x$. Finally, by \L o\'s's Theorem and inductive hypothesis $[u]_\SSS = [\proj_y]_\SSS = y \in x$. 

		\item  Fix $a \in \CCC$. If $y = [\proj_y]_\SSS \in a$, by fineness
		\[
			\bp{ f \in O_{a}: f(y) \in \ran_a(f)} =  \bp{ f \in O_{a}: y \in \dom(f)} \in F_{a}.
		\]
		thus $y = [\proj_y]_\SSS \in [\ran_a]_\SSS$ by \L o\'s's Theorem. Conversely, assume that $u: O_b \to V$ is such that $[u]_\SSS \in [\ran_a]_\SSS$, $b \supseteq a$. By \L o\'s's Theorem, $u$ is regressive on
		\[
			A = \bp{f \in O_b: u(f) \in \ran_a(\pi_a(f)) = f[a]} \in F_b
		\]
		thus by normality there exist $y \in c \supseteq b$ and $B \subseteq \pi_c^{-1}[A]$  such that $u(\pi_b(f)) = f(y)$ for all $f \in B$. Since $B \subseteq \pi_c^{-1}[A]$, $f(y) = u(\pi_b(f)) = f(x)$ for some $x \in a$. Since $f$ is injective, $x = y \in a$ and $[u]_\SSS = [\proj_y]_\SSS = y$ by \L o\'s's Theorem.

		\item Fix $a \in \CCC$. If $x \in a$, by fineness $\bp{f \in O_a : x \in \dom_a(f)}\in F_a$ hence $j(x) = [c_x]_\SSS \in [\dom_a]_\SSS$. Conversely, assume $[u: O_b \to V]_\SSS \in [\dom_a]_\SSS$ with $b \supseteq a$. By \L o\'s's Theorem, $A = \bp{f \in O_b: u(f) \in \dom_a(\pi_a(f))}\in F_b$ and we can define
		\[
		\begin{array}{rrcl}
			v: & A & \longrightarrow  & V \\
			   & f & \longmapsto & f(u(f)).
		\end{array}
		\]
		that is regressive on $A$. Then by normality there exist $y \in c \supseteq b$ and $B \subseteq \pi^{-1}_c[A]$ such that $v(\pi_b(f))= f(u(\pi_b(f))) = f(y)$ for all $f \in B$. Since $f$ is injective, $u(\pi_b(f)) = y$ hence $y$ is in $\dom_a(\pi_a(f)) = \dom(f) \cap a$. Thus by \L o\'s's Theorem $[u]_\SSS= [c_y]_\SSS =j(y)$.
		
		\item Follows from points 3 and 4, together with the observation that $[\id_a]_\SSS$ has to be an ${\in}$-preserving function by \L o\'s's Theorem and $(j \res a)^{-1}$ is the only such function with domain $j[a]$ and range $a$. \qedhere
	\end{enumerate}
\end{proof}

These canonical representatives can be used in order to prove many general properties of $\CCC$-system of filters and of the induced ultrapowers. In particular, we shall use them to prove Propositions \ref{prop:support_monotonicity} and \ref{prop:completeness}, and other related properties.

\begin{proposition} \label{prop:induced_system}
	Let $\SSS$ be a $\CCC$-system of ultrafilters, $A \subseteq O_a$ be such that $a \in \CCC$. Then $A \in F_a$ if and only if $\cp{j_\SSS \res a}^{-1} \in j_\SSS(A)$.
\end{proposition}
\begin{proof}
	By \L o\'s's Theorem, we have $A = \bp{f \in O_a : ~ f = \id_a(f) \in A} \in F_a$ if and only if $\cp{j_\SSS \res a}^{-1} = [\id_a]_\SSS \in [c_A]_\SSS = j_\SSS(A)$.
\end{proof}

\begin{lemma} \label{lem:supports}
	Let $\SS$ be a $\CCC$-system of filters and $a \in \CCC$. Then $\kappa_a$ is the minimum $\alpha$ such that $\Qp{j_{\dot{\ff}(\SS)}(\check{\alpha}) \geq \rank(\check{a})}_\SS = \1$.
\end{lemma}
\begin{proof}
	Let $j$ be the elementary embedding derived from $\dot{\ff}(\SS)$ in a generic extension by $\SS$. Notice that $O_a \cap {}^aV_\alpha \in F_a$ is equivalent by Proposition \ref{prop:induced_system} to
	\[
	\cp{j \res a}^{-1} \in j(O_a \cap {}^aV_\alpha) = j(O_a) \cap {}^{j(a)}V_{j(\alpha)}
	\]
	which is in turn equivalent to $a \subseteq V_{j(\alpha)}$ i.e. $j(\alpha) \geq \rank(a)$. Since this holds in all generic extensions by $\SS$, we are done.	
\end{proof}

\begin{proof}[Proof of Proposition~\ref{prop:support_monotonicity}]
	By the previous proposition $\kappa_a$ is the minimum $\alpha$ such that $\1 \Vdash_\SS j(\alpha) \geq \rank(a)$, hence it depends (monotonically) only on $\rank(a)$. The conclusion of Proposition \ref{prop:support_monotonicity} follows.
\end{proof}

\begin{proposition} \label{prop:critical_point}
	Let $\SS$ be a $\CCC$-system of filters. Then $\kappa$ is the critical point of $j = j_{\dot{\ff}(\SS)}$ with boolean value $\1$ iff $\SS$ is a $\ap{\kappa, \lambda}$-system of filters.
\end{proposition}
\begin{proof}
	Suppose that $\kappa$ is the critical point of $j$ with boolean value $\1$. If $\gamma < \kappa$, $A \in F_{\bp{\gamma}}$ iff $\1 \Vdash_\SS (j \res \bp{\gamma})^{-1} = j(\id \res \bp{\gamma}) \in j(A)$ iff $\id \res \bp{\gamma} \in A$. Thus $F_{\bp{\gamma}}$ is principal generated by $\id \res \bp{\gamma}$.
	If $a \in \CCC \cap V_{\kappa+2}$, $\rank(a) \leq \kappa+1 \leq j_{\dot{\ff}(\SS)}(\kappa)$ with boolean value $\1$, thus $\kappa_a \leq \kappa$ by Lemma \ref{lem:supports}.

	Conversely, suppose that $\bp{\id \res \bp{\gamma}} \in F_{\bp{\gamma}}$ for $\gamma < \kappa$, and $\kappa_a \leq \kappa$ for $a \in V_{\kappa+2}$. If there is an $A \in \SS^+$ forcing that $j$ has no critical point or has critical point bigger than $\kappa$, $\kappa_a = \rank(a) > \kappa$ for $a \in V_{\kappa+2} \setminus V_{\kappa+1}$, a contradiction. 
	If there is a $B \in \SS^+$ forcing that $j$ has critical point $\gamma$ smaller than $\kappa$, $F_{\bp{\gamma}}$ cannot be principal generated by $\id \res \bp{\gamma}$, again a contradiction.
\end{proof}

\begin{proof}[Proof of Proposition \ref{prop:completeness}]
	Let $\SS$ be a $\ap{\kappa, \lambda}$-system of filters, $a$ be in $\CCC$, $j$ be derived from $\dot{\ff}(\SS)$. We need to prove that $F_a$ is ${<}\kappa$-complete for all $a$. Suppose that $\AAA \subseteq F_a$ is such that $\vp{\AAA} < \kappa$. Hence by Proposition \ref{prop:critical_point}, $j(A)=j[A]$. Then $\bigcap \AAA \in F_a$ iff
	\[
	\1 \Vdash_\SS \cp{j \res a}^{-1} \in j(\bigcap \AAA) = \bigcap j(\AAA) = \bigcap j[\AAA]
	\]
	which is true since $A \in F_a \Rightarrow \1 \Vdash_\SS \cp{j \res a}^{-1} \in j(A)$ for all $A \in \AAA$.
\end{proof}

The ultrapower $\Ult(V, \SSS)$ happens to be the direct limit of the directed system of ultrapowers $\ap{\Ult(V, F_a) : a \in \CCC}$ with the following factor maps:
\[
\begin{array}{rrcl}
	k_{ab} : & \Ult(V,F_a) & \longrightarrow & \Ult(V,F_b)\\
		& [u]_{F_a} & \longmapsto & [u \circ \pi_{ba}]_{F_b}
\end{array}
\]
\[
\begin{array}{rrcl}
 	k_{a} : & \Ult(V,F_a) & \longrightarrow & \Ult(V,\SSS)\\
 		& [u]_{F_a} & \longmapsto & [u]_{\SSS}
\end{array}
\]
The ultrapower $\Ult(V, \SSS)$ is also the direct limit of the ultrapowers given by the restrictions of $\SSS$, as shown in the following.

\begin{definition}
	Let $\SS$ be a $\CCC$-system of filters of length $\lambda$, $\alpha < \lambda$ be an ordinal. The restriction of $\SS$ to $\alpha$ is the $(\CCC \cap V_\alpha)$-system of filters $\SS \res \alpha = \bp{F_a : ~ a \in \CCC \cap V_\alpha}$. Moreover, if $\SSS$ is a $\CCC$-system of ultrafilters the corresponding factor map is
	\[
	\begin{array}{rrcl}
	 	k_{\alpha} : & \Ult(V,\SSS \res \alpha ) & \longrightarrow & \Ult(V,\SSS)\\
	 		& [u]_{\SSS \res \alpha} & \longmapsto & [u]_{\SSS}.
	\end{array}
	\]
\end{definition}
 
\begin{proposition}
	Let $\SSS$ be a $\CCC$-system of ultrafilters of length $\lambda$, $\alpha < \lambda$ be an ordinal. Then
	\begin{enumerate}
		\item $k_\alpha$ is elementary;
		\item $k_\alpha \circ j_{\SSS \res \alpha} = j_\SSS$;
		\item $k_\alpha \res \bigcup(\CCC \cap V_\alpha) = \id \res \bigcup(\CCC \cap V_\alpha)$, hence $\crit(k_\alpha) \geq \alpha$.
	\end{enumerate}
\end{proposition}
\begin{proof}
	A particular case of Proposition \ref{prop:factormap} to follow.
%	\begin{enumerate}
%		\item Let $\phi(x_1, \ldots, x_n)$ be a formula, $u_i : O_{a_i} \to V$ for $i \leq n$ be in $U_\SSS$. Define $b = \bigcup\bp{a_i: {1 \leq i \leq n}}$. Since $ \CCC \cap V_\alpha$ is a directed set of domains $b \in \CCC \cap V_\alpha$. Hence $\Ult(V,\SSS \res \alpha) \models \phi([u_1]_{\SSS \res \alpha}, \dots, [u_n]_{\SSS \res \alpha})$ if and only if (by \L o\'s theorem)
%		\[
%		\bp{f \in O_b: ~ \phi(u_1(\pi_{ba_i}(f)),\dots,u_1(\pi_{ba_n}(f)))} \in F_b.
%		\]
%		if and only if (again by \L o\'s theorem) $\Ult(V,\SSS) \models \phi([u_1]_{\SSS}, \dots, [u_n]_{\SSS})$. The thesis follows by definition of $k_\alpha$.
%		\item For any $x \in V$, $k_\alpha(j_{\SSS\res \alpha}(x)) = k_\alpha([c_x]_{\SSS \res \alpha}) = [c_x]_\SSS = j_\SSS(x)$.
%		\item Let $x$ be in $\bigcup(\CCC \cap V_\alpha)$. By Proposition \ref{prop:rappcanonici}, $k_\alpha(x)= k_\alpha([\proj_x]_{\SSS\res \alpha})= k_\alpha([\proj_x]_\SSS) = x$. 
%	\end{enumerate}
\end{proof}

% % % % % % % % % % % % % % % % % % % % % % % % % % % % % % % % % % % % % % % % % % % % % % % % % % %

\subsection{System of ultrafilters derived from an embedding} \label{ssec:sys_derived}

We now present the definitions and main properties of $\CCC$-system of ultrafilters derived from a generic elementary embedding. With abuse of notation, we denote as generic elementary embedding any map $j: V \to M$ which is elementary and such that $M \subseteq W$ for some $W \supseteq V$. In the following we shall assume that $j$ is a definable class in $W$. However, we believe that it should be possible to adapt the present results to non-definable $j$, provided we are working in a strong enough set theory with sets and classes (e.g. $\MK$). We also provide a comparison between derived $\CCC$-systems of ultrafilters for different choices of $\CCC$. 

Let $\SSS$ be a $\CCC$-system of ultrafilters, $A \subseteq O_a$ be such that $a \in \CCC$. Then by Proposition \ref{prop:induced_system},
\[
	A \in F_a  \iff \cp{j_\SSS \res a}^{-1} \in j_\SSS(A)
\]
and this relation actually provides a definition of $\SSS$ from $j_\SSS$. This justifies the following definition.

\begin{definition} \label{def:derived_csu}
	Let $V \subseteq W$ be transitive models of $\ZFC$. Let $j: V \to M \subseteq W$ be a generic elementary embedding definable in $W$, $\CCC \in V$ be a directed set of domains such that for any $a \in \CCC$, $(j\res a)^{-1} \in M$. The \emph{$\CCC$-system of ultrafilters derived from $j$} is $\SSS = \ap{F_a : ~ a \in \CCC}$ such that:
	\[
	F_a = \bp{A \subseteq O_a : \cp{j \res a}^{-1} \in j(A)}.
	\]
\end{definition}

Definition \ref{def:derived_csu} combined with Proposition \ref{prop:induced_system} guarantees that for a given a $\CCC$-system of ultrafilters $\SSS$, the $\CCC$-system of ultrafilters derived from $j_\SSS$ is $\SSS$ itself. We now show that the definition is meaningful for any embedding $j$.

\begin{proposition} \label{prop: derived_system}
	Let $j$, $\CCC$, $\SSS$ be as in the definition above. Then $\SSS$ is a $\CCC$-system of $V$-ultrafilters. 
\end{proposition}
\begin{proof}
\begin{enumerate}
	\item \emph{(Filter and ultrafilter property)} Fix $a \in \CCC$ and assume that $A, B \in F_a$. Then $(j \res a)^{-1} \in j(A) \cap j(B) = j(A \cap B).$ Moreover if $C \subseteq O_a$ and $A \subseteq C$, then $(j \res a)^{-1} \in j(A) \subseteq j(C)$. Finally, if $(j \res a)^{-1} \notin j(A)$ we have that $(j \res a)^{-1} \in j(O_a) \setminus j(A) = j(O_a \setminus A)$. 

	\item \emph{(Fineness)} Fix $x \in a$ so that $j(x) \in j[a]$. Then $j(x) \in \dom((j\res a)^{-1})$ hence we have $\bp{f \in O_a : x \in \dom(f)} \in F_a$ by definition of $F_a$.

	\item \emph{(Compatibility)} Assume that $a \subseteq b \in \CCC$ and $A \subseteq O_a$. Then
	\[
	(j\res b)^{-1} \in j(\pi^{-1}_{ba}[A]) = \bp{f \in O_{j(b)} : \pi_{j(a)}(f) \in j(A)}
	\]
	if and only if $(j \res a)^{-1} = \pi_{j(a)}((j\res b)^{-1}) \in j(A)$.

	\item \emph{(Normality)} Let $u : A \to V$ be regressive on $A \in F_a$ and in $V$. By elementarity,
	\[
		M \models \forall f \in j(A) ~ \exists x \in \dom(f) ~ j(u)(f) \trianglelefteq f(x)
	\]
	Since $(j\res a)^{-1} \in j(A)$, there exists $x \in j[a]$ with $j(u)((j\res a)^{-1}) \trianglelefteq (j\res a)^{-1}(x)$.  Define $y = j(u)((j\res a)^{-1})$, and put $b = a \cup \bp{y}$. Note that $y \trianglelefteq j^{-1}(x) \in a$ hence by transitivity of $\bigcup \CCC$, $\bp{y}\in \CCC$. Define $B = \bp{f \in O_b : u(\pi_{ba}(f)) = f(y)}$.  Then $(j \res b)^{-1} \in j(B)$, i.e. $B \in F_{b}$, since
	\[
	(j \res b)^{-1}(j(y))=y=j(u)((j \res a)^{-1})=j(u)(\pi_{j(a)}((j \res b)^{-1})). \qedhere
	\]
\end{enumerate}
\end{proof}

Given a $\CCC$-system of ultrafilters $\SSS$ derived from a generic embedding $j$, we can factor out the embedding $j$ through $j_\SSS$.

\begin{definition} \label{def:factormap}
	Let $j: V \to M \subseteq W$ be a generic elementary embedding, $\CCC \in V$ be a directed set of domains of length $\lambda$, $\SSS$ be the $\CCC$-system of ultrafilters derived from $j$. Then
	\[
	\begin{array}{rrcl}
		k: & \Ult(V,\SSS) & \longrightarrow & M \\
			& [u: O_a \to V ]_\SSS & \longmapsto & j(u)(\cp{j\res a}^{-1})
	\end{array}
	\]
	is the \emph{factor map associated to $\SSS$}.
\end{definition}

\begin{proposition} \label{prop:factormap}
	Let $j$, $\CCC$, $\lambda$, $\SSS$, $k$ be as in the previous definition. Then
	\begin{enumerate}
		\item $k$ is elementary;
		\item $k \circ j_\SSS= j$;
		\item $k \res \bigcup \CCC = \id \res \bigcup \CCC$ hence $\crit(k) \geq \lambda$;
		\item if $\lambda = j(\gamma)$ for some $\gamma$, then $\crit(k) > \lambda$.
	\end{enumerate}
\end{proposition}
\begin{proof}
	\begin{enumerate}
		\item Let $\phi(x_1, \dots, x_{n})$ be a formula, and for any $i\in n$ let $u_i: O_{a_i} \to V$, $a_i \in \CCC$. Put $b =\bigcup \bp{a_i : 1 \leq i \leq n}$ . Then $\Ult(V,\SSS) \models \phi([u_1]_\SSS, \dots, [u_{n}]_\SSS)$ if and only if (by \L o\'s Theorem)
		\[
		B = \bp{f \in O_b : \phi(u_1(\pi_{a_1}(f)), \dots, u_{n}(\pi_{a_{n}}(f)))} \in F_b.
		\]
		if and only if $(j \res b)^{-1} \in j(B)$ (by definition of $F_b$) i.e.
		\[
		M \models \phi(j(u_1)(\pi_{j(a_1)}(j \res b)^{-1}), \dots, j(u_{n})(\pi_{j(a_{n})}(j \res b)^{-1}))
		\]
		if and only if (by definition of $\pi_{j(a_i)}$)
		\[
		M \models \phi(j(u_1)((j \res a_i)^{-1}), \dots, j(u_{n})((j \res a_n)^{-1})).
		\]
		i.e. $M \models \phi(k([u_1]_\SSS), \dots, k([u_{n}]_\SSS)$.

		\item For any $x \in V$,
		\[
		k(j_\SSS(x)) = k([c_x]_\SSS) = j(c_x)(\emptyset) = c_{j(x)}(\emptyset) = j(x).
		\]

		\item Let $x \in \bigcup \CCC$. Then by Proposition \ref{prop:rappcanonici} for some $a\in \CCC$ with $x\in a$,
		\[
		k(x) = k([\proj_x]_\SSS) = j(\proj_x)(\cp{j\res a}^{-1}) = j^{-1}(j(x)) = x.
		\]

		\item If $\lambda = j(\gamma)$, the following diagram commutes:
		\begin{center}
			\begin{tikzpicture}[xscale=1.8,yscale=1.2]
					\node (A0_0) at (0, 0) {$V$};
					\node (A1_0) at (1, 0) {$M$};
					\node (A2_0) at (2, 0) {$W$};
					\node (A1_1) at (1, -1) {$\Ult(V,\SSS)$};
					\node (A2_1) at (2, -1) {$V[\SSS]$};
					\path (A0_0) edge [->] node [auto] {$j$} (A1_0);
					\path (A0_0) edge [->] node [auto,swap] {$j_\SSS$} (A1_1);
					\path (A1_1) edge [->] node [auto,swap] {$k$} (A1_0);
					\path (A1_0) edge[draw=none] node [sloped] {$\subseteq$} (A2_0);
					\path (A1_1) edge[draw=none] node [sloped] {$\subseteq$} (A2_1);
					\path (A2_0) edge[draw=none] node [sloped] {$\subseteq$} (A2_1);
			\end{tikzpicture}
		\end{center}
		Thus $\crit(k) \geq j(\gamma)$ and $k \circ j_\SSS(\gamma) = j(\gamma)$. Therefore $j(\gamma) \in \ran(k)$ hence $j(\gamma)$ cannot be the critical point of $k$, showing that the above inequality is strict.\footnote{Remark that in the above diagram $V[\SSS]$ is the smallest transitive model $N$ of $\ZFC$ such that $V,\bp{\SSS}\subseteq N\subseteq W$.} \qedhere
	\end{enumerate}
\end{proof}

Observe that in Definition \ref{def:derived_csu} $\cp{j \res a}^{-1} \notin M$ would imply that the derived filter $F_a$ is empty. Thus, depending on the choice of $\CCC$, there can be a limit on the maximal length attainable for a $\CCC$-system of ultrafilter derived from $j$. If $\CCC = \qp{\lambda}^{{<}\omega}$, $\cp{j \res a}^{-1}$ is always in $M$ thus there is no limit on the length of the extenders derived from $j$. If $\CCC = V_\lambda$, the maximal length is the minimal $\lambda$ such that $j[V_\lambda] \notin M$. These bounds are relevant, as shown in the following proposition.

\begin{proposition} \label{prop:tallest_tower}
	Let $\TTT \in W$ be a tower of length a limit ordinal $\lambda$, $j : V \to M \subseteq W$ be the derived embedding. Then the tallest tower derivable from $j$ is $\TTT$.
\end{proposition}
\begin{proof}
	Since $\dom_{V_\alpha}$ represents $j[V_\alpha]$ for all $\alpha < \lambda$, $j[V_\alpha] \in M$ for all $\alpha < \lambda$ and we only need to prove that $j[V_\lambda] \notin M$. Suppose by contradiction that $u : O_a \to V$ is such that $[u]_\TTT = j[V_\lambda]$. Let $\alpha < \lambda$ be such that $a \in V_\alpha$, and let $v : O_a \to V$ be such that $v(x) = u(x) \cap V_{\alpha+1}$. Thus $[v]_\TTT = j[V_\lambda] \cap j(V_{\alpha+1}) = j[V_{\alpha+1}]$,  and by \L o\'s Theorem
	\[
	A = \bp{f \in O_{V_{\alpha+1}} : ~ v(\pi_a(f)) = \dom(f)} \in F_{V_{\alpha+1}}
	\]
	Since $\vp{\dom[A]} \leq \vp{\ran(v)} \leq \vp{O_a} < \vp{V_{\alpha+1}}$, $\dom[A]$ is a non-stationary subset of $V_{\alpha+1}$ by Lemma \ref{lem:small_ns} contradicting Proposition \ref{prop:largeness}.
\end{proof}

We now consider the relationship between different $\CCC$-systems of ultrafilters derived from a single $j$.

\begin{proposition}\label{prop:s1_subset_s2}
	Let $j: V \to M \subseteq W$ be a generic elementary embedding definable in $W$, $\CCC_1 \subseteq \CCC_2$ be directed  sets of domains in $V$, $\SSS_n$ be the $\CCC_n$-system of $V$-ultrafilters derived from $j$ for $n = 1, 2$. Then $\Ult(V,\SSS_2)$ can be factored into $\Ult(V,\SSS_1)$, and $\crit(k_1) \leq \crit(k_2)$ where $k_1$, $k_2$ are the corresponding factor maps.
\end{proposition}
\begin{proof}
	We are in the following situation:
	\begin{center}
		\begin{tikzpicture}[xscale=3, yscale=1.2]
				\node (V) at (0, 0) {$V$};
				\node (M) at (2, 0) {$M$};
				\node (N1) at (1, -1.7) {$\Ult(V,\SSS_1)$};
				\node (N2) at (1, -0.5) {$\Ult(V, \SSS_2)$};
				\path (V) edge [->] node [auto] {$j$} (M);
				\path (V) edge [->] node  [auto,swap] {$j_1$} (N1);
				\path (V) edge [->] node  [auto,swap, pos = 0.8 ] {$j_2$} (N2);
				\path (N1) edge [->] node [auto,swap] {$k_1$} (M);
				\path (N2) edge [->] node [auto,swap, pos = 0.2 ] {$k_2$} (M);
				\path (N1) edge [->] node [auto,swap] {$k$} (N2);
		\end{tikzpicture}
	\end{center}
	where $k$ is defined as
	\[
	\begin{array}{rrcl}
		k: & \Ult(V,\SSS_1) & \longrightarrow & \Ult(V, \SSS_2) \\
		   & [u]_{\SSS_1} & \longmapsto & [u]_{\SSS_2}
	\end{array}
	\]
	Observe that $j_1$, $j_2$ and $k$ commute. Moreover given $u: O_a \to V$ with $a \in \CCC_1$,
	\[
	k_2 \circ k([u]_{\SSS_1}) = j(u)\cp{\cp{j \res a}^{-1}} = k_1([u]_{\SSS_1})
	\]
	therefore the diagram commutes. Since $k_1$ and $k_2$ are elementary, $k$ has to be elementary as well and $\crit(k_1) \leq \crit(k_2)$.
\end{proof}

Notice that the last proposition can be applied whenever $\SSS_1$ is an extender and $\SSS_2$ is a tower, both of the same length $\lambda$ and derived from the same generic elementary embedding $j : V \to M \subseteq W$. It is also possible for a ``thinner'' system of filters (i.e. an extender) to factor out a ``fatter'' one.

\begin{definition}
	Let $F$ be an ultrafilter. We denote by $\non(F)$ the minimum of $\vp{A}$ for $A \in F$. Let $\SSS$ be a $\CCC$-system of ultrafilters. We denote by $\non(\SSS)$ the supremum of $\non(F_a)+1$ for $a \in \CCC$.
\end{definition}

If the length of $\SSS$ is a limit ordinal $\lambda$, $\non(\SSS)$ is bounded by $\beth_\lambda$. If $\EEE$ is a $\gamma$-extender of regular length $\lambda > \gamma$, $\non(\SSS)$ is also bounded by $2^{{<}\lambda} + 1$.

\begin{theorem} \label{thm:e_dominates_s}
	Let $\CCC$ be a directed set of domains. Let $j: V \to M \subseteq W$ be a generic elementary embedding definable in $W$, $\SSS$ be the $\CCC$-system of filters derived from $j$, $\EEE$ be the extender of length $\lambda \supseteq j[\non(\SSS)]$ derived from $j$. Then $\Ult(V,\EEE)$ can be factored into $\Ult(V,\SSS)$, and $\crit(k_\SSS) \leq \crit(k_\EEE)$.
\end{theorem}
\begin{proof}
	Let $\rho_a : \qp{\non(F_a)}^1 \to O_a$ be an enumeration of an $A \in F_a$ of minimum cardinality, so that $\cp{j \res a}^{-1} \in j(A) = \ran(j(\rho_a))$. Let $k$ be defined by
	\[
	\begin{array}{rrcl}
		k: & \Ult(V,\SSS) & \longrightarrow & \Ult(V, \EEE) \\
		   & [u : O_a \to V]_{\SSS} & \longmapsto & [u \circ \rho_a \circ \ran_{\bp{\beta}}]_{\EEE}
	\end{array}
	\]
	where $\beta < j(\non(F_a)) \leq \lambda$ is such that $j(\rho_a)(\bp{\beta}) = \cp{j \res a}^{-1}$. We are in the following situation:
	\begin{center}
		\begin{tikzpicture}[xscale=3, yscale=1.2]
				\node (V) at (0, 0) {$V$};
				\node (M) at (2, 0) {$M$};
				\node (N1) at (1, -1.7) {$\Ult(V,\SSS)$};
				\node (N2) at (1, -0.5) {$\Ult(V, \EEE)$};
				\path (V) edge [->] node [auto] {$j$} (M);
				\path (V) edge [->] node [auto,swap] {$j_\SSS$} (N1);
				\path (V) edge [->] node [auto,swap, pos = 0.8] {$j_\EEE$} (N2);
				\path (N1) edge [->] node [auto,swap] {$k_\SSS$} (M);
				\path (N2) edge [->] node [auto,swap, pos= 0.2] {$k_\EEE$} (M);
				\path (N1) edge [->] node [auto,swap] {$k$} (N2);
		\end{tikzpicture}
	\end{center}
	Observe that $j_\SSS$, $j_\EEE$ and $k$ commute. Moreover given $u: O_a \to V$ with $a \in \CCC$,
	\begin{align*}
		k_\EEE \circ k([u]_{\SSS}) &= j(u \circ \rho_a \circ \ran_{\bp{\beta}})\cp{\cp{j \res \bp{\beta}}^{-1}} \\
		&= j(u \circ \rho_a)(\bp{\beta}) \\
		&= j(u)\cp{\cp{j \res a}^{-1}} = k_\SSS([u]_\SSS)
	\end{align*}
	therefore the diagram commutes. Since $k_\SSS$ and $k_\EEE$ are elementary, $k$ has to be elementary as well and $\crit(k_\SSS) \leq \crit(k_\EEE)$.
\end{proof}

The last proposition with $j = j_\SSS$ shows that from any $\CCC$-system of filters $\SSS$ can be derived an extender $\EEE$ of sufficient length such that $\Ult(V, \SSS) = \Ult(V, \EEE)$. The derived extender $\EEE$ might have the same length as $\SSS$, e.g. when $\lambda = \beth_\lambda$ and $j[\lambda] \subseteq \lambda$. In particular, this happens in the notable case when $\SSS$ is the full stationary tower of length $\lambda$ a Woodin cardinal.

	\section{Generic large cardinals} \label{sec:generic_lc}
		% !TEX root = GenericLargeCardinals.tex

Generic large cardinal embeddings are analogous to classical large cardinal embeddings. The difference between the former and the latter is that the former is definable in some forcing extension of $V$ and not in $V$ itself as the latter. An exhaustive survey on this topic is given in \cite{foreman:generic_embeddings}. Most of the large cardinal properties commonly considered can be built from the following basic blocks.

\begin{definition}\label{definition:glce}
	Let $V\subseteq W$ be transitive models of $\ZFC$. Let $j: V \to M \subseteq W$ be a generic elementary embedding with critical point $\kappa$. We say that
	\begin{itemize}
		\item $j$ is $\gamma$-tall iff $j(\kappa) \geq \gamma$;\footnote{We remark that the present definition of $\gamma$-tall for an embedding does not coincide with the classical notion of $\gamma$-tall for cardinals, which is witnessed (in the present terms) by a \emph{$\gamma$-tall and $\kappa$-closed} embedding with critical point $\kappa$. We believe that the present definition is more convenient to our purposes since it avoids overlapping of concepts and simplifies the corresponding combinatorial version for $\CCC$-systems of filters.}
		\item $j$ is $\gamma$-strong iff $V_\gamma^W \subseteq M$;
		\item $j$ is ${<}\gamma$-closed iff ${}^{{<}\gamma}M \subseteq M$ from the point of view of $W$.
	\end{itemize}
\end{definition}

Notice that the definition of a large cardinal property through the existence of an embedding $j$ with (some version of) the above properties \emph{is not a first-order statement}, since it quantifies over a \emph{class} object. In the theory of large cardinals in $V$, this problem is overcome by showing that an extender $\EEE$ of sufficient length is able to capture all the aforementioned properties of $j$ in $j_\EEE$. For generic large cardinals the same can be done with some additional limitations, as shown in Section \ref{ssec:lc_firstorder}.

In contrast with the classical case, this process requires the use of $\CCC$-systems of ultrafilters \emph{in some generic extension}. However, it would be a desirable property to be able to obtain a description of such generic elementary embeddings from objects living in $V$. Natural intuition suggest the feasibility of this option (see e.g. \cite{larson:stationary_tower}). We thus introduce the following definition schema (already suggested in \cite{claverie:ideal_extenders}).
 
\begin{definition}[Claverie] \label{def:ideally_lc}
	Let $P$ be a large cardinal property of an elementary embedding (i.e. a first-order property in the class parameter $j$), $\kappa$ be a cardinal. We say that $\kappa$ has property $P$ iff there exists an elementary embedding $j : V \to M \subseteq V$  with critical point $\kappa$ and satisfying property $P$.

	We say that $\kappa$ has \emph{generically} property $P$ iff there exists a forcing extension $V[G]$ and an elementary embedding $j : V \to M \subseteq V[G]$  definable in $V[G]$ and with critical point $\kappa$ satisfying property $P$.

	We say that $\kappa$ has \emph{ideally} property $P$ iff there exist a $\CCC$-system of filters $\SS$ in $V$ such that the corresponding generic ultrapower embedding $j_{\dot{\ff}(\SS)}$ satisfies property $P$ in the corresponding generic extension.
\end{definition}

Observe that for any $\kappa$, $P(\kappa) \Rightarrow \text{ideally } P(\kappa) \Rightarrow \text{generically } P(\kappa)$. On the other side, it is not clear whether $\text{generically } P(\kappa) \Rightarrow \text{ideally } P(\kappa)$ as pointed out in \cite{claverie:ideal_extenders, cody:cox:generically_strong}. In \cite{a:thesis} an example is given suggesting that the natural procedure of inducing a $\CCC$-system of filters in $V$ from a generic elementary embedding might fail to preserve large cardinal properties, thus giving some hints against the equivalence of these two concepts.

\begin{theorem}[\cite{a:thesis}] \label{thm:failure}
	Let $\delta$ be a Woodin cardinal. Then for any $\kappa \in [\omega_1, \delta)$ there is a generically superstrong embedding $j$ with critical point $\kappa$ such that the tallest tower derivable from $j$ embeds in the original forcing in a densely incomplete way.
\end{theorem}

Furthermore, having ideally property $P$ can be much weaker than having property $P$ in $V$: e.g. the consistency of an ideally $I_1$ cardinal follows from the consistency of a Woodin cardinal \cite{larson:stationary_tower}. Nonetheless, upper bounds on the consistency of generic large cardinals similar to those for classical large cardinals can be proven (see \cite{hamkins:kunen_inconsistency, suzuki:non_existence}), e.g. the inconsistency of a set-generic Reinhardt cardinal.
Since from a stationary tower of height a Woodin cardinal we can obtain a \emph{class}-generic Reinhardt cardinal, it is clear that the strength of a generic large cardinal very much depends on the nature of the forcing allowed to obtained it. In fact, the strength of a generic large cardinal hypothesis depends on the interaction of \emph{three parameters}, as outlined in \cite{foreman:generic_embeddings}: the size of the critical point, the closure properties of the embedding, and the nature of the forcing used to define it.
We shall not expand on the impact of the nature of forcing, while we shall spend some time on the \emph{size} of the critical point. In this setting, the trivial observation that $P(\kappa) \Rightarrow \text{ideally } P(\kappa) \Rightarrow \text{generically } P(\kappa)$ is not really satisfying, since we are interested in the consistence of \emph{small} cardinals $\kappa$ having ideally (or generically) property $P$.
However, it is sometimes possible to collapse a large cardinal in order to obtain a small generic large cardinal. Examples of positive results on this side can be found in \cite{claverie:ideal_extenders, cody:cox:generically_strong, foreman:quotient_embeddings, kakuda:generic_precipitous, magidor:precipitous}, we present and generalize some of them in Section \ref{ssec:consistency}.

Notice that having ideally property $P$ is inherently a statement on the structure of the relevant $\CCC$-system of filters $\SS$ in question. In Section \ref{ssec:combinatoric} we provide a characterization of these properties as combinatorial statements on $\SS$.

Since having a generic large cardinal property is possibly weaker than having the same property in $V$, two large cardinal properties which are inequivalent for classical large cardinals may turn out to be equivalent for their generic counterparts. In Section \ref{ssec:distinction_glcp} we show some examples of embeddings separating different generic large cardinal properties. These examples are an application of the techniques introduced throughout all this section.

% % % % % % % % % % % % % % % % % % % % % % % % % % % % % % % % % % % % % % % % % % % % % % % % % % %

\subsection{Deriving large cardinal properties from generic systems of filters} \label{ssec:lc_firstorder}

All over this section $G$ is $V$-generic for some forcing $\BB$ and $j: V \to M \subseteq V[G]$ is a generic elementary embedding definable in $V[G]$ with some large cardinal property $P$ and critical point $\kappa$. We aim to approximate $j$ via a suitable $\CCC$-system of $V$-ultrafilters $\SSS$ in $V[G]$ (with $\CCC\in V$) closely enough so as to preserve the large cardinal property in question. 

\begin{proposition}
	Let $j$ be $\gamma$-tall, $\CCC\in V$ be a directed set of domains with $\lambda\subseteq\bigcup\CCC$, $\SSS$ be the $\CCC$-system of ultrafilters of length $\lambda \geq j(\kappa)$ derived from $j$. Then $j_\SSS$ is $\gamma$-tall.
\end{proposition}
\begin{proof}
	By Proposition \ref{prop:factormap}, $\crit(k) > j(\kappa)$ hence $j_\SSS(\kappa) = k(j_\SSS(\kappa)) = j(\kappa) > \gamma$.
\end{proof}

\begin{proposition}\label{prop:sys_superstrong}
	Let $j$ be $\gamma$-strong and $\lambda$ be such that either $\lambda > \gamma$ or $\lambda = j(\mu) = \gamma$ for some $\mu$. Let $\CCC\in V$ be a directed set of domains with $\lambda\subseteq\bigcup\CCC$, and $\SSS$ be the $\CCC$-system of ultrafilters derived from $j$. Then $j_\SSS$ is $\gamma$-strong; i.e. $V_\gamma^{\Ult(V,\SSS)} = V_\gamma^{V[\SSS]} = V_\gamma^{V[G]}$.
\end{proposition}
\begin{proof}
	By Proposition \ref{prop:factormap} we have that $\crit(k) > \gamma$. Thus
	\[
	V_\gamma^{\Ult(V,\SSS)} = k(V_\gamma^{\Ult(V,\SSS)}) = V_\gamma^M = V_\gamma^{V[G]}.
	\]
	Furthermore, since $V_\gamma^{\Ult(V,\SSS)} \subseteq V_\gamma^{V[\SSS]} \subseteq V_\gamma^{V[G]}$ they must all be equal.\footnote{Once again $V[\SSS]$ is the minimal transitive model $N$ of $\ZFC$ such that $V,\bp{\SSS}\subseteq N\subseteq V[G]$.}
\end{proof}

While tallness and strongness are easily handled, in order to ensure preservation of closure we need some additional technical effort.

\begin{definition} \label{def:presaturation}
	A boolean algebra $\BB$ is ${<}\lambda$-presaturated is for any $\gamma < \lambda$ and family $\AAA = \ap{A_\alpha: ~ \alpha < \gamma}$ of maximal antichains of size $\lambda$, there are densely many $p \in \BB^+$ such that
	\[
	\forall \alpha < \gamma ~ \vp{\bp{a \in A_\gamma : ~ a \wedge p > 0}} < \lambda.
	\]
\end{definition}

\begin{proposition}
	Let $\lambda$ be a regular cardinal. A boolean algebra $\BB$ is ${<}\lambda$-presaturated if and only if it preserves the regularity of $\lambda$.
\end{proposition}

\begin{lemma} \label{lem:sequenzegeneriche}
	Let $\SSS$ be a $\CCC$-system of ultrafilters in $V[G]$ and $N = \Ult(V,\SSS)$ be such that:
	\begin{itemize}
		\item $V[G]$ is a ${<}\lambda^+$-cc forcing extension for some $\lambda$ regular in $V[G]$;
		\item $V_\lambda^{V[G]} = V_\lambda^N$ and $\CCC$ has length at least $\lambda$;
		\item $N$ is closed for ${<}\lambda$-sequences in $V$.
	\end{itemize}
	Then $N$ is closed for ${<}\lambda$-sequences in $V[G]$.
\end{lemma}
\begin{proof}
	Let $\dot{s}$ be the name for a sequence of length $\gamma < \lambda$ of elements of $N$. Since $\lambda$ is regular in $V[G]$, the forcing $\CC$ which defines $V[G]$ is ${<}\lambda$-presaturated. Moreover, $\CC$ is ${<}\lambda^+$-cc hence for any $\alpha < \gamma$ there are at most $\lambda$-many possibilities for $\dot{s}(\alpha)$. Therefore we can apply presaturation and find a condition $p \in G$ such that 
	\[
	p \Vdash \dot{s} = \bp{\ap{\ap{\alpha,[u^\alpha_\beta]_{\dot{\SSS}}},q^\alpha_\beta}: \alpha < \gamma, \beta < \mu}, 
	\]
	for some $\mu < \lambda$.
	
	Let $\ap{x_\alpha : ~ \alpha < \lambda}$ be a (partial) enumeration of $\CCC \cap V_\lambda$. Define $t: \gamma \times \mu \to V$ so that $t(\alpha, \beta) = \ap{u^\alpha_\beta, \ran_{\bp{x_\alpha}}, \ran_{\bp{x_\beta}}}$ is a sequence in $V$. Since $N$ is closed for sequences in $V$, the sequence represented by $t$ is in $N$; i.e. 
	\[
	X = \bp{\ap{[u^\alpha_\beta]_\SSS,\ap{\bp{x_\alpha}, \bp{x_\beta}}}: ~ \alpha < \gamma, \beta < \delta} \in N.
	\] 
	Moreover, $Y\in V[G]$ where $Y = \bp{\ap{\bp{x_\alpha}, \bp{x_\beta}}: ~ q^\alpha_\beta \in G}$. Since $Y \in V_\lambda^{V[G]} = V_\lambda^N$, inside $N$ we can define $\val_G(\dot{s}) = \bp{\ap{\alpha,[u^\alpha_\beta]_\SSS} \in N: ~ \exists y \in Y ~ \ap{u^\alpha_\beta,y} \in X}$.
\end{proof}

\begin{theorem} \label{thm:sys_closure}
	Let $j$ be ${<}\lambda$-closed with $\lambda$ regular cardinal, $V[G]$ be a ${<}\lambda^+$-cc forcing extension. Let $\CCC\in V$ be a ${<}\lambda$-directed set of domains in $V$, let $\SSS$ be the $\CCC$-system of filters derived from $j$. Then $j_\SSS$ is ${<}\lambda$-closed.
\end{theorem}
\begin{proof}
	Let $u_\alpha: O_{a_\alpha} \to V$ for $\alpha < \gamma$ be a sequence of length $\gamma < \lambda$ in $V$ of elements of $\Ult(V, \SSS)$. Since $\CCC$ is ${<}\lambda$-directed, there is a $b \supseteq \bigcup\bp{ a_\alpha : {\alpha < \gamma}} \in \CCC$ such that $\vp{b} \geq \gamma$. Let $\ap{x_\alpha : ~ \alpha < \gamma}$ be a (partial) enumeration of $b$. Define
	\begin{align*}
		v: O_b &\to V \\
		f &\mapsto \bp{u_\alpha(f): ~ x_\alpha \in \dom(f)}
	\end{align*}
	so that by fineness and normality $[v]_\SSS = \bp{[u_\alpha]_\SSS: ~ \alpha < \gamma}$. Thus $\Ult(V, \SSS)$ is closed under ${<}\lambda$-sequences in $V$ and is $\lambda$-strong by Proposition \ref{prop:sys_superstrong}. We can apply Lemma \ref{lem:sequenzegeneriche} to obtain that $j_\SSS$ is ${<}\lambda$-closed.
\end{proof}

Note that since in the hypothesis of the previous theorem $j$ is ${<}\lambda$-closed, it is always possible to derive a system of filters with a ${<}\lambda$-directed set of domains. In particular, it is possible to derive towers of length $\lambda$ and $\lambda$-extenders of any length. Thus the only significant limitation is the hypothesis that $V[G]$ is ${<}\lambda^+$-cc where $\lambda$ is the amount of closure required for $M$. However, this hypothesis is satisfied in most classical examples of generic elementary embeddings with high degrees of closure, as e.g. the full stationary tower of length a Woodin cardinal.

It is also possible to ensure the same closure properties with non-directed system of filters, as e.g. extenders. This can be done by means of Theorem \ref{thm:e_dominates_s} and the following remarks.

% % % % % % % % % % % % % % % % % % % % % % % % % % % % % % % % % % % % % % % % % % % % % % % % % % %

\subsection{Consistency of small generic large cardinals} \label{ssec:consistency}

In this section we shall prove that in most cases the assertion that a small cardinal (e.g. $\omega_1$) has generically or ideally property $P$ consistently follows from the existence of any such cardinal (Corollary \ref{cor:small_glc}). Similar results were proved independently in \cite{kakuda:generic_precipitous, magidor:precipitous} and echoed by \cite{claverie:ideal_extenders}; we generalize them to $\CCC$-system of filters, obtaining a simpler proof.

In the following we shall need to lift embeddings and systems of filters in forcing extensions. We refer to \cite[Chp. 9]{cummings:iterated_embeddings} for a complete treatment of the topic. Recall that if $j : V \to M$ is an elementary embedding and $\BB \in V$ is a boolean algebra, $j$ is also an elementary embedding of the boolean valued model $V^\BB$ into $M^{j(\BB)}$. Furthermore, $j$ can be lifted to the generic extensions $V[G]$ and $M[H]$ where $G$ is $V$-generic for $\BB$ and $H$ is $j(\BB)$-generic for $M$ whenever $j[G] \subseteq H$.

For sake of simplicity, we shall focus on the boolean valued models approach and avoid explicit use of generic filters. This will be convenient to handle several different forcing notions at the same time. All the proofs will then be carried out in $V$ using names and explicitly mentioning in which boolean valued model $V^\BB$ every sentence is to be interpreted.

\begin{definition}
	Let $\BB$ be a complete boolean algebra, and $\dot{\CC}$ be a $\BB$-name for a complete boolean algebra. We denote by $\BB \ast \dot{\CC}$ the boolean algebra defined in $V$ whose elements are the equivalence classes of $\BB$-names for elements of $\dot{\CC}$ (i.e. $\dot{a}\in V^{\BB}$ such that $\Qp{\dot{a}\in\dot{\CC}}_{\BB}=\1$) modulo the equivalence relation $\dot{a} \approx \dot{b} ~ \Leftrightarrow ~ \Qp{\dot{a} = \dot{b}}_\BB = \1$.
\end{definition}

We refer to \cite{a:viale:notes_forcing, a:viale:semiproper_iterations} for further details on two-step iterations and iterated forcing.

\begin{definition}
	Let $\SS$ be a $\CCC$-system of filters, $\CC$ be a cBa. Then $\SS^\CC = \bp{F_a^\CC : ~ a \in \CCC}$ where $F_a^\CC = \bp{A \subseteq \cp{O_a}^{V^\CC} : ~ \exists B \in \check{F}_a ~ A \supseteq B}$.
\end{definition}

We remark that the following theorem is built on the previous results by Kakuda and Magidor \cite{kakuda:generic_precipitous, magidor:precipitous} for single ideals and by Claverie \cite{claverie:ideal_extenders} for ideal extenders.

\begin{theorem} \label{thm:coll_embedding}
	Let $j : V \to M \subseteq V^\BB$ be elementary with critical point $\kappa$, and $\CC \in V$ be a ${<}\kappa$-cc cBa. Then $\BB \ast j(\CC)$ factors into $\CC$, and the embedding $j$ lifts to $j^\CC : V^\CC \to M^{j(\CC)}$.
	\begin{center}
	\begin{tikzpicture}[xscale=2,yscale=1]
			\node (V) at (0, 1) {$V$};
			\node (M) at (1, 1) {$M$};
			\node (VB) at (2, 1) {$V^\BB$};
			\node (VC) at (0, 0) {$V^{\CC}$};
			\node (MC) at (1, 0) {$M^{j(\CC)}$};
			\node (VBC) at (2, 0) {$V^{\BB \ast j(\CC)}$};
			\path (V) edge [->] node [auto] {$j$} (M);
			\path (VC) edge [->] node [auto,swap] {$j^\CC$} (MC);
			\path (V) edge[draw=none] node [sloped] {$\subseteq$} (VC);
			\path (VB) edge[draw=none] node [sloped] {$\subseteq$} (VBC);
			\path (M) edge[draw=none] node [sloped] {$\subseteq$} (MC);
			\path (M) edge[draw=none] node [sloped] {$\subseteq$} (VB);
			\path (MC) edge[draw=none] node [sloped] {$\subseteq$} (VBC);
	\end{tikzpicture}
	\end{center}

	Furthermore, if $\BB = \SS = \ap{F_a : ~ a \in \CCC}$ is a $\ap{\kappa, \lambda}$-system of filters and $j = j_{\dot{\ff}(\SS)}$, then $\CC \ast \SS^\CC$ is isomorphic to $\SS \ast j(\CC)$ and $j^\CC$ is the embedding induced by $\SS^\CC$.
	\begin{center}
	\begin{tikzpicture}[xscale=2,yscale=1]
			\node (V) at (0, 1) {$V$};
			\node (M) at (1, 1) {$M$};
			\node (VB) at (2, 1) {$V^\SS$};
			\node (VC) at (0, 0) {$V^{\CC}$};
			\node (MC) at (1, 0) {$M^{j(\CC)}$};
			\node (VCB) at (2, 0) {$V^{\SS \ast j(\CC)}$};
			\node (VBC) at (3, 0) {$V^{\CC \ast \SS^\CC}$};
			\path (V) edge [->] node [auto] {$j_{\dot{\ff}(\SS)}$} (M);
			\path (VC) edge [->] node [auto,swap] {$j_{\dot{\ff}(\SS^\CC)}$} (MC);
			\path (V) edge[draw=none] node [sloped] {$\subseteq$} (VC);
			\path (VB) edge[draw=none] node [sloped] {$\subseteq$} (VCB);
			\path (VCB) edge[draw=none] node [sloped] {$=$} (VBC);
			\path (M) edge[draw=none] node [sloped] {$\subseteq$} (MC);
			\path (M) edge[draw=none] node [sloped] {$\subseteq$} (VB);
			\path (MC) edge[draw=none] node [sloped] {$\subseteq$} (VCB);
	\end{tikzpicture}
	\end{center}
\end{theorem}
\begin{proof}
	For the first part, consider the embedding:
	\[
	\begin{array}{rrcl}
		i_1: & \CC & \longrightarrow  & \BB \ast j(\CC) \\
		   & p & \longmapsto & j(p)
	\end{array}
	\]
	By elementarity of $j$, $i_1$ must preserve $\leq$, $\perp$. Given any maximal antichain $\AAA$, $\CC$ is ${<}\kappa$-cc hence $j[\AAA] = j(\AAA)$ which is maximal again by elementarity of $j$. Then $i_1$ is a complete embedding hence $\BB \ast j(\CC)$ is a forcing extension of $\CC$. Thus we can lift $j$ to a generic elementary embedding $j^\CC$.

	For the second part, consider the embedding:
	\[
	\begin{array}{rrcl}
		i_2: & \CC \ast \SS^\CC & \longrightarrow  & \SS \ast j(\CC) \\
		   & \dot{A} \subseteq O_a &\longmapsto & \Qp{[\id_a]_{\dot{\ff}(\SS)} \in j(\dot{A})}_{\SS \ast j(\CC)}
	\end{array}
	\]
	This map is well-defined since the set of $\dot{A} \in \CC \ast \SS^\CC$ such that $\dot{A} \subseteq O_a$ for some fixed $a \in \CCC$ is dense in $\CC \ast \SS^\CC$. Suppose now that $\dot{A} \leq_{\CC \ast \SS^\CC} \dot{B}$ with $\dot{B} \subseteq O_b$, $b \in \CCC$, $c = a \cup b$. Then,
	\begin{align*}
		\1 &\Vdash_\CC \cp{\pi^{-1}_c[\dot{A}] \setminus \pi^{-1}_c[\dot{B}]} \in I_c^\CC \quad\Rightarrow \\
		\1 &\Vdash_\CC \exists C \in I_c ~ \cp{\pi^{-1}_c[\dot{A}] \setminus \pi^{-1}_c[\dot{B}]} \subseteq C
	\end{align*}
	and we can find a maximal antichain $\AAA \subseteq \CC$ such that $p \Vdash_\CC \cp{\pi^{-1}_c[\dot{A}] \setminus \pi^{-1}_c[\dot{B}]} \subseteq \check{C}_p$ for every $p \in \AAA$ and corresponding $C_p \in I_c \Rightarrow \1 \Vdash_\CC [\id_c]_{\dot{\ff}(\SS)} \notin j(C_p)$. Thus by elementarity of $j$, for all $p \in \AAA$ we have that
	\[
		j(p) \Vdash_{j(\CC)} \cp{\pi^{-1}_c[j(\dot{A})] \setminus \pi^{-1}_c[j(\dot{B})]} \subseteq j(\check{C}_p) \not\ni [\id_c]_{\dot{\ff}(\SS)} 
	\]
	and since $j[\AAA]$ is maximal in $j(\CC)$,
	\begin{align*}
		& \1 \Vdash_{j(\CC)} [\id_{a \cup b}]_{\dot{\ff}(\SS)}  \notin \cp{\pi^{-1}_c[j(\dot{A})] \setminus \pi^{-1}_c[j(\dot{B})]} \quad\Rightarrow \\
		& \1 \Vdash_{j(\CC)} [\id_b]_{\dot{\ff}(\SS)}  \in j(\dot{B}) \vee [\id_a]_{\dot{\ff}(\SS)}  \notin j(\dot{A}) \quad\Rightarrow \\
		& i_2(\dot{B}) \vee \neg i_2(\dot{A}) = \1 \Rightarrow i_2(\dot{A}) \leq i_2(\dot{B})
	\end{align*}
	Thus $i_2$ preserves $\leq$. Preservation of $\perp$ is easily verified by a similar argument, replacing everywhere $\dot{A} \setminus \dot{B}$ with $\dot{A} \cap \dot{B}$.

	We still need to prove that $i_2$ has a dense image. Fix $[u^p]_{\dot{\ff}(\SS)} \in \SS \ast j(\CC)$, so that $u^p : A \to \CC$, $A \in I_a^+$, $a \in \CCC$. Let $\dot{B} = \bp{x \in \check{A} : ~ \check{u}^p(x) \in \dot{G}_\CC}$ be in $V^\CC$. Then,
	\begin{align*}
		i_2(\dot{B}) &= \Qp{[\id_a]_{\dot{\ff}(\SS)}  \in j(\check{A}) ~\wedge~ j(\check{u}^p)([\id_a]_{\dot{\ff}(\SS)} ) \in j(\dot{G}_\CC)}_{\SS \ast j(\CC)} \\
		&= \Qp{\check{A} \in \dot{\ff}(\SS) ~\wedge~ [\check{u}^p]_{\dot{\ff}(\SS)} \in \dot{G}_{j(\CC)}}_{\SS \ast j(\CC)} = [u^p]_{\dot{\ff}(\SS)}
	\end{align*}
	hence $V^{\SS \ast j(\CC)} = V^{\CC \ast \SS^\CC}$ is the forcing extension of $V^\CC$ by $\SS^\CC$.

	Finally, we prove that $j^\CC$ is the generic ultrapower embedding derived from $\SS^\CC$. We can directly verify that $\SS^\CC$ satisfies filter property, fineness and compatibility. This is sufficient to define an ultrapower $N = \Ult(V^\CC, \dot{\ff}(\SS^\CC))$ and prove \L o\'s Theorem for it. The elements of $N$ are represented by $\CC$-names for functions $\dot{v}: O_a^{V^\CC} \to V^\CC$. Since $F_a^\CC$ concentrates on $O_a^V$ for all $a \in \CCC$, we can assume that $\dot{v}: \check{O}_a \to V^\CC$. Furthermore, we can replace $\dot{v}$ by a function $u: O_a \to V^\CC$ in $V$ mapping $f \in O_a$ to a name for $\dot{v}(\check{f})$. These functions can then represent both all elements of $N$ and all elements of $M^{j(\CC)}$. Furthermore, $N$ and $M^{j(\CC)}$ must give the same interpretation to them. In fact, given $u_n : O_a \to V^\CC$ for $n = 1, 2$ and $\dot{A} \in \cp{I_a^\CC}^+$:
	\begin{align*}
		& \dot{A} \Vdash_{\CC \ast \SS^\CC} [u_1]_{\dot{\ff}(\SS^\CC)} = [u_2]_{\dot{\ff}(\SS^\CC)} \iff \\
		& \1 \Vdash_\CC \bp{f \in \dot{A} : ~ u_1(f) \neq u_2(f)} \in I_a^\CC \iff \\
		& \exists B \in I_a ~ \1 \Vdash_\CC \forall f \in \dot{A} \setminus \check{B} ~ \dot{v}_1(f) = \dot{v}_2(f) \iff \\
		& \exists B \in F_a ~ \1 \Vdash_\CC \forall f \in \dot{A} \cap \check{B} ~ \dot{v}_1(f) = \dot{v}_2(f) \iff \\
		& \exists B \in F_a ~ \forall f \in B ~ \1 \Vdash_\CC f \in \dot{A} \rightarrow u_1(f) = u_2(f) \iff \\
		& \bp{f \in O_a : ~ \1 \Vdash_\CC f \in \dot{A} \rightarrow u_1(f) = u_2(f)} \in F_a \iff \\
		& \1 \Vdash_{\SS \ast j(\CC)} [\id_a]_{\dot{\ff}(\SS)} \in j(\dot{A}) \rightarrow [u_1]_{\dot{\ff}(\SS)} = [u_2]_{\dot{\ff}(\SS)} \iff \\
		& i_2(\dot{A}) = \Qp{[\id_a]_{\dot{\ff}(\SS)} \in j(\dot{A})}_{\SS \ast j(\CC)} \Vdash_{\SS \ast j(\CC)} [u_1]_{\dot{\ff}(\SS)} = [u_2]_{\dot{\ff}(\SS)}
	\end{align*}
	and the above reasoning works also replacing $=$ with $\in$. The second passage uses essentially that $\CC$ is ${<}\kappa$-cc and $\SS$ is a $\ap{\kappa, \lambda}$-system of filters. In fact, in this setting given $\dot{A} \in I_a^\CC$ there are less than $\kappa$ possibilities for a $B \in I_a$, $p \Vdash \check{B} \supseteq (\dot{A} \cap \check{O}_a)$, hence we can find a single such $B$ by ${<}\kappa$-completeness of $I_a$ (see Proposition \ref{prop:completeness}).
\end{proof}

\begin{corollary}
	Let $\SS$ be a $\ap{\kappa, \lambda}$-system of filters, $\CC$ be a ${<}\kappa$-cc cBa. Then $\SS^\CC$ is a $\CCC$-system of filters.
\end{corollary}
\begin{proof}
	Since $\SS^\CC$ is the $\CCC_\SS$-system of filters in $V$ derived from $j_{\dot{\ff}(\SS)}^\CC$, it is a $\CCC$-system of filters by Propositions \ref{prop: derived_system} and \ref{prop:ideal_generic}.
\end{proof}

\begin{proposition} \label{prop:coll_closure_strength}
	Let $j : V \to M \subseteq V^\BB$ be elementary with critical point $\kappa$, $\gamma < \kappa$ be a cardinal, and $j^\CC : V^\CC \to M^{j(\CC)}$ be obtained from $j$ and $\CC = \Coll(\gamma, {<}\kappa)$. Suppose that $j(\kappa)$ is regular in $V^\BB$.

	If $j$ is ${<}\delta$-closed with $\delta \geq j(\kappa)$, then $j^\CC$ is ${<}\delta$-closed. If $j$ is $\delta$-strong with $\delta \geq j(\kappa)$, then $j^\CC$ is $\delta$-strong.
\end{proposition}
\begin{proof}
	Since $j(\kappa)$ is regular in $V^\BB$, $\Coll(\gamma, {<}j(\kappa))$ is ${<}j(\kappa)$-cc in $V^\BB$. Moreover, the order on the L\'evy collapse is absolute between transitive models thus $j(\CC) = \Coll(\gamma, {<}j(\kappa))^M$ is a suborder of $\Coll(\gamma, {<}j(\kappa))$. Hence $j(\CC)$ is also ${<}j(\kappa)$-cc in $V^\BB$.

	First, suppose that $j$ is ${<}\delta$-closed and let $\sigma$ be a $j(\CC)$-name for a sequence of ordinals of size $\mu < \delta$. Since $\sigma(i)$ for $i < \mu$ is decided by an antichain of size less than $j(\kappa)$, the whole $\sigma$ is coded by a subset of $M$ of size less than $\delta + j(\kappa) = \delta$. Thus $\sigma \in M$ hence is evaluation is in $M^{j(\CC)}$.
	
	Suppose now that $j$ is $\delta$-strong and let $\sigma$ be a $j(\CC)$-name for a subset of $\mu < \delta$. Then $\sigma$ is coded by a subset of $M$ of size less than $\delta + j(\kappa) = \delta$ as before, hence $\sigma$ is in $M$ and its evaluation in $M^{j(\CC)}$.
\end{proof}

\begin{corollary} \label{cor:small_glc}
	Let $P$ be a property among $(n)$-huge, almost $(n)$-huge (for $n > 0$), $\alpha$-superstrong (for $\alpha > \kappa$), $(n)$-superstrong (for $n > 1$).

	If $\kappa$ is generically (resp. ideally) $P$, then it is so after $\Coll(\gamma, {<}\kappa)$ for any $\gamma < \kappa$. Thus the existence of a generically (resp. ideally) $P$ cardinal is equiconsistent with $\omega_1$ being such a cardinal.
\end{corollary}

Note that the previous corollary applies only to generically and ideally $P$: the existence of a large cardinal with property $P$ in $V$ is usually stronger than $\omega_1$ being generically $P$. Due to the fact that a generically superstrong cardinal does not guarantee that $j(\kappa)$ is regular in $V^\BB$, the previous result does not apply to superstrong cardinals. We recall that the case of a strong cardinal was already treated in \cite[Corollary 4.14]{claverie:ideal_extenders}, which showed the following.

\begin{theorem}
	The existence of a strong cardinal is equiconsistent with $\omega_1$ being ideally strong.
\end{theorem}

As in Proposition \ref{prop:coll_closure_strength}, it is possible to prove that forcing with $\Coll(\gamma, {<}\kappa)$ with $\kappa$ a strong cardinal preserves the ideally strongness of $\kappa$. However, starting with an ideally strong cardinal would not suffice in this case. In order to get a $j^\CC$ with strength $\gamma$ we need an embedding $j : V \to M \subseteq V^\BB$ with enough strength so as to contain in $M$ a name for $V_\gamma^\BB$. Although, since the complexity of such a name depends on $\BB$, and $\BB$ depends on the amount of strength that we wish to achieve, there is no hope to sort out this circular reference. On the other hand, a generically strong cardinal is preserved under \emph{Cohen} forcing under some assumptions \cite{cody:cox:generically_strong}.

Notice that the previous corollary does not apply also to generically supercompact cardinals. However, this is not surprising since $\kappa = \gamma^+$ being generically supercompact is equivalent to being generically almost huge: in fact, if $j: V \to M \subseteq V[G]$ is a $\gamma$-closed embedding obtained by $\gamma$-supercompactness, it is also almost huge since $j(\kappa) = (\gamma^+)^{V[G]}$. Thus such a preservation theorem for supercompactness would in turn imply the equiconsistency of generically supercompactness and generically almost hugeness, which is not expected to hold. However, if we restrict the class of forcing to \emph{proper} forcings, is possible to obtain a similar preservation theorem \cite{foreman:quotient_embeddings}.

% % % % % % % % % % % % % % % % % % % % % % % % % % % % % % % % % % % % % % % % % % % % % % % % % % %

\subsection{Combinatorial equivalents of ideally large cardinal properties} \label{ssec:combinatoric}

The \emph{ideal} properties of cardinals given in Definition \ref{def:ideally_lc} are inherently properties of a $\CCC$-system of filters, it is therefore interesting to reformulate them in purely combinatorial terms. In this section we review the main results on this topic present in literature, adapted to the paradigm introduced in Section \ref{sec:system_filters}; and we integrate them with a characterization of strongness that, to our knowledge, is not yet present in literature.

\subsubsection{Critical point and tallness}

In order to express any large cardinal property, we need to be able to identify the critical point of an embedding $j_{\dot{\ff}(\SS)}$ derived from some $\CCC$-system of filters $\SS$.

\begin{definition}
	Let $\SS$ be a $\CCC$-system of filters in $V$. The \emph{completeness} of $\SS$ is the minimum of the completeness of $F_a$ for $a \in \CCC$, i.e. the unique cardinal $\kappa$ such that every $F_a$ is ${<}\kappa$-complete and there is an $F_a$ that is not ${<}\kappa^+$-complete.

	We say that $\SS$ has \emph{densely} completeness $\kappa$ iff it has completeness $\kappa$ and there are densely many $B \in \SS^+$ disproving ${<}\kappa^+$-completeness (i.e. that are the union of $\kappa$ sets in the relevant ideal).
\end{definition}

\begin{proposition}
	Let $\SS$ be a $\CCC$-system of filters in $V$. Then the following are equivalent:
	\begin{enumerate}
		\item the ultrapower map $\dot{k} = j_{\dot{\ff}(\SS)}$ has critical point $\kappa$ with boolean value $\1$;
		\item $\SS$ is a $\ap{\kappa, \lambda}$-system of filters;
		\item $\SS$ has densely completeness $\kappa$;
	\end{enumerate}
	Moreover, if $\kappa \in \CCC$ then the statements above are also equivalent to
	\begin{enumerate}
		\item[4.] $\bp{\id \res \alpha : ~ \alpha < \kappa} \in F_\kappa$.
	\end{enumerate}
\end{proposition}
\begin{proof}
	$(1) \Leftrightarrow (2)$: has already been proved in Proposition \ref{prop:critical_point}.

	$(2) \Rightarrow (3)$: By Proposition \ref{prop:completeness}, we know that $F_a$ is ${<}\kappa$-complete for all $a \in \CCC$. Let $\dot{u}$ be a name for a function representing $\kappa$ in $\Ult(V, \dot{\ff}(\SS))$. Then there are densely many $B \in I_b^+$ deciding that $\dot{u} = \check{v}$, for some $v : B \to \kappa$. Since 
	\[
	\Qp{[v]_{\dot{\ff}(\SS)} \neq \alpha = [\dom_\alpha]_{\dot{\ff}(\SS)}}_\SS \geq B
	\]
	for all $\alpha < \kappa$, $B_\alpha = B \wedge v^{-1}[\bp{\alpha}] \in I_b$ for any such $B$ hence $B = \bigcup_{\alpha < \kappa} B_\alpha$ disproves ${<}\kappa^+$-completeness.

	$(3) \Rightarrow (1)$: We prove by induction on $\alpha < \kappa$ that $\1 \Vdash_\SS j(\check{\alpha}) = \check{\alpha}$. Let $u : A \to \alpha$ with $A \in I_a^+$ be representing an ordinal smaller than $j(\alpha)$ in the ultrapower, and let $A_\beta = u^{-1}\qp{\bp{\beta}}$ for $\beta < \alpha$. Since $A = \bigcup_{\beta < \alpha} A_\beta$ and $\SS$ is ${<}\kappa$-complete, the conditions $A_\beta$ form a maximal antichain below $A$ hence $[u]_{\dot{\ff}(\SS)}$ is forced to represent some $\beta < \alpha$. Furthermore, there are densely many $B \in I_b^+$ that are a union of $\kappa$-many sets $B_\alpha \in I_b$. From any one of them we can build a function $u : B \to \kappa$, $u(f) = \alpha_f$ where $f \in B_{\alpha_f}$, so that $B$ forces that $[u]_{\dot{\ff}(\SS)} < j(\kappa)$ and $[u]_{\dot{\ff}(\SS)} > \alpha$ for all $\alpha < \kappa$. Thus $B \Vdash_\SS j(\kappa) > \kappa$ for densely many $B$.

	Assume now that $\kappa \in \CCC$. Then $(1) \Leftrightarrow (4)$ follows from Proposition \ref{prop:rappcanonici} and \L o\'s theorem, since $\bp{\id \res \alpha : ~ \alpha < \kappa}$ is equal to
	\[
		\bigwedge_{\alpha < \kappa} \Qp{[\ran_\alpha]_{\dot{\ff}(\SS)} = j(\alpha)}_\SS \wedge \Qp{[\ran_\kappa]_{\dot{\ff}(\SS)} < j(\kappa)}_\SS = \Qp{j[\kappa] = \kappa \wedge j(\kappa) > \kappa}_\SS. \qedhere
	\]
\end{proof}

A similar approach can apply also to tallness-related properties.

\begin{proposition}
	Let $\SS$ be a $\ap{\kappa, \lambda}$-system of filters in $V$. The ultrapower map $j = j_{\dot{\ff}(\SS)}$ is $\gamma$-tall for $\gamma < \lambda$ iff $\bp{f \in O_{\bp{x}} : \rank(f(x)) \leq \kappa} \in F_{\bp{x}}$ for some $x \in \bigcup \CCC$ with $\rank(x) = \gamma$.
\end{proposition}
\begin{proof}
	By Proposition \ref{prop:rappcanonici} and \L o\'s theorem the above set is equal to
	\[
	\Qp{\gamma = \rank(x) = \rank([\proj_x]_{\dot{\ff}(\SS)}) \leq j(\kappa)}_\SS. \qedhere
	\]
\end{proof}

\subsubsection{Measurability}

We say that a cardinal is measurable iff there is an elementary embedding $j : V \to M$ with critical point $\kappa$ such that the image is well-founded. Its generic counterpart can be characterized for $\CCC$-systems of filters by means of the following definition.

\begin{definition}
	Let $\SS$ be a $\CCC$-system of filters in $V$. We say that $\SS$ is \emph{precipitous} iff for every $B \in \SS^+$ and sequence $\ap{\AAA_\alpha : ~ \alpha < \omega} \in V$ of maximal antichains in $<_\SS$ below $B$, there are $\bar{A}_\alpha \in \AAA_\alpha$, $\bar{A}_\alpha \in I_{\bar{a}_\alpha}^+$ and $h : \bigcup_\alpha \bar{a}_\alpha \to V$ such that $\pi_{\bar{a}_\alpha}(h) \in \bar{A}_\alpha$ for all $\alpha < \omega$.
\end{definition}

This definition is equivalent to \cite[Def. 4.4.\textit{ii}]{claverie:ideal_extenders} for ideal extenders, and to ${<}\omega$-closure for extenders in $V$ (see \cite{koellner:large_cardinals}), while being applicable also to other systems of filters. The results relating these definitions with well-foundedness are subsumed in the following.

\begin{theorem}
	Let $\SS$ be a $\CCC$-system of filters in $V$. The ultrapower map $j = j_{\dot{\ff}(\SS)}$ is well-founded iff $\SS$ is precipitous.
\end{theorem}
\begin{proof}
	First, suppose that $\SS$ is precipitous and assume by contradiction that $B$ forces the ultrapower to be ill-founded. Let $\ap{\dot{u}_n : ~ n < \omega}$ be $\SS$-names for functions $\dot{u}_n : O_{\dot{a}_n} \to V$ in $U_\SS$ such that $\Qp{[\dot{u}_{n+1}]_{\dot{\ff}(\SS)} \in [\dot{u}_n]_{\dot{\ff}(\SS)}}_\SS \geq B$. Define $\dot{b}_n = \bigcup \bp{ \dot{a}_m: {m \leq n}}$, $\dot{B}_0 = O_{\dot{b}_0}$, and
	\[
	\dot{B}_{n+1} = \bp{x \in O_{\dot{b}_{n+1}} : ~ \dot{u}_{n+1}(\pi_{\dot{a}_{n+1}}(x)) \in \dot{u}_n(\pi_{\dot{a}_n}(x))}
	\]
	so that $\Qp{\dot{B}_n \in \dot{\ff}(\SS)}_\SS \geq B$. Fix $n < \omega$. By the forcing theorem there is a dense set of $A$ in $\SS$ below $B$ deciding the values of $\dot{u}_n$, $\dot{B}_n$; and every such $A \Vdash \dot{B}_n = \check{B}_n$ must force that $\check{B}_n \in \dot{\ff}(\SS)$ hence satisfy $A <_\SS B_n$. It follows that the set of $A \in I_a^+$ deciding $\dot{u}_n$, $\dot{B}_n$ and with the additional property that every such $A$ satisfy $a \supseteq b_n$, $A \subseteq \pi_a^{-1}[B_n]$, is also dense below $B$. Let $\AAA_n$ be a maximal antichain below $B$ in this set.

	Let $\bar{A}_n$, $\bar{a}_n$, $h:\bar{a}=\cup_{n}\bar{a}_n$ be obtained from $\ap{\AAA_n : ~ n < \omega}$ by precipitousness of $\SS$. Let also $u_n$, $B_n$ be such that $\bar{A}_n \Vdash \dot{u}_n = \check{u}_n \wedge \dot{B}_n = \check{B}_n$. Then $\pi_{\bar{a}\bar{a}_n}(h) \in \bar{A}_n \subseteq \pi_{\bar{a}_n}^{-1}[B_n]$ and $\pi_{\bar{a}b_n}(h) \in B_n$ for all $n < \omega$. Thus, $u_{n+1}(\pi_{a_{n+1}}(h)) \in u_n(\pi_{a_n}(h))$ is an infinite descending chain in $V$, a contradiction.

	Suppose now that $\SS$ is not precipitous, and fix $B$, $\ap{\AAA_n : ~ n < \omega}$ witnessing it. Define a tree $T$ of height $\omega$ consisting of couples of sequences $\ap{\BBB, f}$ such that $\BBB_n \in \AAA_n$ for all $n < \vp{\BBB} < \omega$ and $f \in \bigwedge \BBB$, ordered by member-wise inclusion. Since $\ap{\AAA_n : ~ n < \omega}$ contradicts precipitousness, the tree $T$ has no infinite chain and we can define a rank-like function $r : T \to \ON$ by well-founded recursion on $T$ as $r(x) = \bigcup \bp{r(y)+1 : ~ y <_T x}$. Notice that $y <_T x$ implies $r(y) < r(x)$.

	Let $\BBB$ be as above, and define $u_\BBB : \bigwedge \BBB \to V$ by $u_\BBB(f) = r(\ap{\BBB, f})$. Let $\dot{u}_n$ be the $\SS$-name defined by $\dot{u}_n = \bp{\ap{\check{u}_\BBB, \bigwedge \BBB} : ~ \BBB \in \Pi_{m \leq n} \AAA_m }$. Then any $\BBB \in \Pi_{m \leq n+1} \AAA_m$ forces $\dot{u}_{n+1}$ to be $u_\BBB$, $\dot{u}_n$ to be $u_{\BBB \res n}$, and $\dot{u}_{n+1} \in \dot{u}_n$ since for all $f \in \bigwedge \BBB$,
	\[
	u_\BBB(f) = r(\ap{\BBB, f}) < r(\ap{\BBB \res n, f'}) = u_{\BBB \res n}(f')
	\]
	where $f' = f \res \dom(\bigwedge (\BBB \res n))$. Since $\bp{ \bigwedge \BBB : ~ \BBB \in \Pi_{m \leq n+1} \AAA_m}$ forms a maximal antichain below $B$ for every $i$, $B$ forces that $\ap{\dot{u}_n : ~ n < \omega}$ is a name for an ill-founded chain.
\end{proof}

\subsubsection{Strongness}

In this section we cover large cardinal properties defined in terms of the existence of elementary embeddings $j : V \to M \subseteq V[G]$ with certain degree of strongness (i.e. such that $V^{V[G]}_\gamma \subseteq M$ for some appropriate $\gamma$). Main examples of such properties are strongness, superstrongness and variants of them. We now present a criterion to characterize $\gamma$-strongness for an elementary embedding $j_{\dot{\ff}(\SS)}$, which can in turn be applied in order to characterize all of the aforementioned large cardinal properties. To our knowledge, there is no equivalent version of the content of this section in the classical tower or extender setting.

\begin{definition}
	Let $\SS$ be a $\CCC$-system of filters, $\AAA_0 \cup \AAA_1$ be an antichain in $\SS^+$. We say that $\ap{\AAA_0, \AAA_1}$ \emph{is split by} $\SS$ iff there exist a $b \in \CCC$ and $B_0, B_1$ disjoint in $\PPP(O_b)$ such that $A \leq_\SS B_n$ for all $A \in \AAA_n$, $n < 2$.

	We say that a family of antichains $\ap{\AAA_{\alpha0} \cup \AAA_{\alpha1} : ~ \alpha < \mu}$ is \emph{simultaneously} split by $\SS$ iff there is a single $b \in \CCC$ witnessing splitting for all of them.
\end{definition}

\begin{definition}
	Let $\SS$ be a $\CCC$-system of filters. We say that $\SS$ is ${<}\gamma$-\emph{splitting} iff for all sequences $\ap{\AAA_{\alpha0} \cup \AAA_{\alpha1} : ~ \alpha < \mu}$ of maximal antichains with $\mu < \gamma$, there are densely many $B \in \SS^+$ such that the antichains $\ap{\AAA_{\alpha0} \res B, \AAA_{\alpha1} \res B}$ for $\alpha < \mu$ are simultaneously split by $\SS$.
\end{definition}

\begin{theorem}[A., S., Viale] \label{thm:strongness_splitting}
	Let $\SS$ be a ${<}\gamma$-directed $\CCC$-system of filters. Then the ultrapower $\Ult(V, \dot{\ff}(\SS))$ contains $\PPP^{V^\SS}(\mu)$ for all $\mu < \gamma$ iff $\SS$ is ${<}\gamma$-splitting.
\end{theorem}
\begin{proof}
	Let $a \in \CCC$, $u_\alpha : O_a \to \ON$ be such that $\Qp{ [\check{u}_\alpha]_{\dot{\ff}(\SS)} = \check{\alpha} }_\SS = \1$ for all $\alpha < \gamma$.
	First, suppose that $\SS$ is ${<}\gamma$-splitting and let $\dot{X}$ be a name for a subset of $\mu < \gamma$. Let $\AAA_{\alpha0} \cup \AAA_{\alpha1}$ for $\alpha < \mu$ be a maximal antichain deciding whether $\check{\alpha} \in \dot{X}$ and $\SSS$ be generic for $\SS$. By ${<}\gamma$-splitting let $B \in \SSS$ be such that $a \subseteq b \in \CCC$,  $B \subseteq O_b$ and $\ap{\AAA_{\alpha0} \res B, \AAA_{\alpha1} \res B}$ is split by $\SS$ in $B_{\alpha0}, B_{\alpha1}$ partitioning $B$ for all $\alpha < \mu$. Then we can define
	\[
	\begin{array}{rrcl}
		v: & B & \longrightarrow  & \PPP(\ON) \\
		   & f & \longmapsto & \bp{u_\alpha(\pi_{ba}(f)) : f \in B_{\alpha1}, ~ \alpha \in \mu}
	\end{array}
	\]
	Then $B$ forces that $[v]_{\dot{\ff}(\SS)} = \dot{X}$, and $B \in \SSS$ so $\val(\dot{X}, \SSS) = [v]_\SSS$ is in $\Ult(V, \SSS)$.

	Suppose now that $\Ult(V, \dot{\ff}(\SS))$ contains $\PPP^{V[\dot{\ff}(\SS)]}(\mu)$ for all $\mu < \gamma$, and let $\ap{\AAA_{\alpha0} \cup \AAA_{\alpha1} : ~ \alpha < \mu}$ be maximal antichains with $\mu < \gamma$. Let $\dot{X} = \bp{\ap{\check{\alpha}, A} : ~ A \in \AAA_{\alpha1}}$ be the corresponding name for a subset of $\mu$, and let $B$, $v : B \to \PPP(\ON)$ be such that $B \Vdash [v]_{\dot{\ff}(\SS)} = \dot{X}$. Let $B_{\alpha0} = \bp{f \in B : ~ u_\alpha(\pi_{ba}(f)) \in v(f)}$, $B_{\alpha1} = B \setminus B_{\alpha0}$. Then $\ap{\AAA_{\alpha0} \res B, \AAA_{\alpha1} \res B}$ is split by $B_{\alpha0}, B_{\alpha1}$ partitioning $B$ for all $\alpha < \mu$.
\end{proof}

\begin{corollary}
	Let $\gamma$ be a limit ordinal, $\SS$ be a ${<}\beth_\gamma$-directed $\CCC$-system of filters. Then the ultrapower $\Ult(V, \dot{\ff}(\SS))$ is $\gamma$-strong iff $\SS$ is ${<}\beth_\gamma$-splitting.
\end{corollary}
\begin{proof}
	If follows by Theorem \ref{thm:strongness_splitting}, together with the observation that in every $\ZFC$ model there is a bijection between elements of $V_\gamma$ and subsets of $\beth_\alpha$ for $\alpha < \gamma$. Such bijection codes $x \in V_\gamma$ as the transitive collapse of a relation on $\vp{\trcl(x)} \leq \beth_\alpha$ for some $\alpha < \gamma$, which in turn is coded by a subset of $\beth_\alpha$.
\end{proof}

\subsubsection{Closure}

In this section we cover large cardinal properties defined in terms of the existence of elementary embeddings $j : V \to M \subseteq V[G]$ with certain degree of closure (i.e. such that ${}^{{<}\gamma} M \subseteq M$ for some appropriate $\gamma$). Main examples of such properties are supercompactness, hugeness and variants of them. We now present a criterion to characterize ${<}\gamma$-closure for an elementary embedding $j_{\dot{\ff}(\SS)}$, which can in turn be applied in order to characterize all of the aforementioned large cardinal properties.

\begin{definition}
	Let $\SS$ be a $\CCC$-system of filters, $\AAA = \bp{A_\alpha : ~ \alpha < \delta}$ be an antichain in $\SS^+$. We say that $\AAA$ \emph{is guessed by} $\SS$ iff there exist a $b \in \CCC$ and $\BBB = \bp{B_\alpha : ~ \alpha < \delta}$ antichain in $\PPP(O_b)$ such that $A_\alpha =_\SS B_\alpha$ for all $\alpha < \delta$.

	We say that a family of antichains $\ap{\AAA_\alpha : ~ \alpha < \mu}$ is \emph{simultaneously} guessed by $\SS$ iff there is a single $b \in \CCC$ witnessing guessing for all of them.
\end{definition}

\begin{definition}
	Let $\SS$ be a $\CCC$-system of filters. We say that $\SS$ is ${<}\gamma$-\emph{guessing} iff for all sequences $\ap{\AAA_\alpha : ~ \alpha < \mu}$ of maximal antichains with $\mu < \gamma$, there are densely many $B \in \SS^+$ such that the antichains $\AAA_\alpha \res B$ for $\alpha < \mu$ are simultaneously guessed by $\SS$ .
\end{definition}

Notice that if an antichain is guessed by $\SS$, every partition of it is split by $\SS$. It follows that ${<}\gamma$-guessing implies ${<}\gamma$-splitting. Furthermore, if $\TT$ is a tower of inaccessible length $\lambda$, ${<}\lambda$-guessing as defined above is equivalent to ${<}\lambda$-presaturation for the boolean algebra $\ap{\TT^+, \leq_\TT}$.

\begin{theorem}
	Let $\lambda$ be an inaccessible cardinal, $\SS$ be a ${<}\lambda$-directed $\CCC$-system of filters of length $\lambda$, $\gamma < \lambda$ be a cardinal. Then the ultrapower $\Ult(V, \dot{\ff}(\SS))$ is ${<}\gamma$-closed iff $\SS$ is ${<}\gamma$-guessing.
\end{theorem}
\begin{proof}
	Let $a \in \CCC$, $u_\alpha : O_a \to \ON$ be such that $\Qp{ [\check{u}_\alpha]_{\dot{\ff}(\SS)} = \check{\alpha} }_\SS = \1$ for all $\alpha < \gamma$.
	First, suppose that $\SS$ is ${<}\gamma$-guessing and let $\dot{s}$ be a name for a sequence $\dot{s} : \mu \to \Ult(V, \dot{\ff}(\SS))$ for some $\mu < \gamma$. Let $\AAA_\alpha$ for $\alpha < \mu$ be a maximal antichain deciding the value of $\dot{s}(\check{\alpha})$, so that given any $A \in \AAA_\alpha$, $A \Vdash \dot{s}(\check{\alpha}) = [\check{v}_A]_{\dot{\ff}(\SS)}$ for some $v_A : O_{a_A} \to V$. Let $\SSS$ be generic for $\SS$. Then by ${<}\gamma$-guessing there is a $B \in \SSS$ such that $\AAA_\alpha \res B$ is guessed by $\SS$ in $\BBB_\alpha \subseteq \PPP(O_b)$ for all $\alpha < \mu$. Furthermore, there can be only $\vp{b} < \lambda$ elements $A \in \AAA_\alpha$ such that the corresponding $A' \in \BBB_\alpha$ is not empty. Since $\SS$ is ${<}\lambda$-directed, there is a single $c \in \CCC$, $c \supseteq a,b$, such that $a_A \subseteq c$ for any $A \in \AAA_\alpha$ that is guessed in an $A' \neq \emptyset$.

	Let $B_\alpha^f$ denote the unique element of $\BBB_\alpha$ such that $\pi_{b}(f) \in B_\alpha^f$, and $v_\alpha^f : O_{a_\alpha^x} \to V$ be such that $B_\alpha^f \Vdash \dot{s}(\check{\alpha}) = [\check{v}^f_\alpha]_{\dot{\ff}(\SS)}$. Then for any $\alpha < \mu$ we can define
	\[
	\begin{array}{rrcl}
		v'_\alpha: & O_c & \longrightarrow  & V \\
		   & f & \longmapsto & v^f_\alpha(\pi_{a_\alpha^f}(f))
	\end{array}
	\]
	so that $B \Vdash [\check{s}]_{\dot{\ff}(\SS)}(\check{\alpha}) = [\check{v}'_\alpha]_{\dot{\ff}(\SS)}$. Since all the $v'_\alpha$ have the same domain $O_c$, we can glue them together forming a single function
	\[
	\begin{array}{rrcl}
		v': & O_c & \longrightarrow  & V \\
		   & f & \longmapsto & \bp{\ap{u_i(\pi_a(f)), v'_\alpha(f)} : ~ \alpha < u_\mu(\pi_a(f))}
	\end{array}
	\]
	Then $B$ forces that $[v']_{\dot{\ff}(\SS)} = \dot{s}$, and $B \in \SSS$ so $\val(\dot{s}, \SSS) = [v']_\SSS$ is in $\Ult(V, \SSS)$.

	Suppose now that $\SS$ is ${<}\gamma$-closed and let $\ap{\AAA_\alpha : ~ \alpha < \mu}$, $\AAA_\alpha = \ap{A_{\alpha\beta} : ~ \beta < \xi_\alpha}$ be as in the definition of ${<}\gamma$-guessing. Let $\dot{f}$ be a name for a sequence $\dot{s} : j_{\dot{\ff}(\SS)}[\mu] \to \ON$ such that $\Qp{\dot{s}(j_{\dot{\ff}(\SS)}(\check{\alpha})) = j_{\dot{\ff}(\SS)}(\check{\beta})}_\SS = [A_{\alpha\beta}]_\SS$. Since $j_{\dot{\ff}(\SS)}$ is ${<}\gamma$-closed, we can find densely many $B \subseteq O_b$ for $b \in \CCC$, $v : B \to \ON^\mu$ such that $B \Vdash \dot{s} = [\check{v}]_{\dot{\ff}(\SS)}$. Then given any $\alpha < \mu$, $\beta < \xi_i$ we can define $B_{\alpha\beta} = \bp{f \in B : ~ v(f)(\alpha) = \beta}$ witnessing guessing for $\ap{\AAA_\alpha \res B : ~ \alpha < \mu}$.
\end{proof}

The above result gives a good characterization of ${<}\gamma$-closure for ideal towers, since ideal towers $\TT$ of inaccessible height $\lambda$ are always ${<}\lambda$-directed. On the other hand, this result does not apply to ideal extenders since their associated system of domains is never ${<}\omega_1$-directed. Since it is not known whether there is such a characterization of ${<}\gamma$-closure for extenders in $V$, we cannot expect to have one in the more general case of ideal extenders.

We can also determine an upper bound to the amount of closure that a system of filters might possibly have.

\begin{definition}
	Let $\SS$ be a $\CCC$-system of filters of length $\lambda$, $\alpha < \lambda$ be an ordinal. We say that $\SS \res \alpha$ \emph{does not express} $\SS$ iff $\1 \Vdash_\SS M \supsetneq M_\alpha$ where $M = \Ult(V, \dot{\ff}(\SS))$, $M_\alpha = \Ult(V, \dot{\ff}(\SS \res \alpha))$.
\end{definition}

Notice that $\1 \Vdash_\SS M \supsetneq M_\alpha$ is equivalent to $\1 \Vdash_\SS M \supsetneq k_\alpha[M_\alpha]$. In fact, $M = M_\alpha$ implies that $k_\alpha = \id \res M$ by Kunen's inconsistency, and $M = k_\alpha[M_\alpha]$ implies that $k_\alpha$ has no critical point thus is the identity. 

\begin{theorem} \label{thm:not_closure}
	Let $\SS$ be a $\CCC$-system of filters of length $\lambda$ such that $\SS \res \alpha$ does not express $\SS$ for any $\alpha < \lambda$. Then $M = \Ult(V, \dot{\ff}(\SS))$ is not closed under $\cof(\lambda)$-sequences.
\end{theorem}
\begin{proof}
	Let $\ap{\xi_\alpha : ~ \alpha < \gamma = \cof(\lambda)}$ be a cofinal sequence in $\lambda$. For all $\alpha < \gamma$, let $\dot{u}_\alpha$ be a name for an element of $M \setminus k_{\xi_\alpha}[M_{\xi_\alpha}]$, $M_{\xi_\alpha} = \Ult(V, \dot{\ff}(\SS \res \xi_\alpha))$. Let $\dot{s}$ be a name such that 	
	\[
	\Qp{\dot{s}(j_{\dot{\ff}(\SS)}(\alpha))=[\dot{u}_\alpha]_{\dot{\ff}(\SS)}}_\SS = \1
	\]
	for all $\alpha < \gamma$, i.e. $\dot{s}$ is a name for a $\gamma$-sequence of elements of $M$ indexed by $j_{\dot{\ff}(\SS)}[\gamma]$.

	Suppose by contradiction that $M$ is closed under $\gamma$-sequences, so that $\1 \Vdash_\SS \dot{s} \in \dot{M}$, then there is an $A \in \SS^+$, $v : A \to V$ such that $A \Vdash_\SS \dot{s} = [\check{v}]_{\dot{\ff}(\SS)}$. Let $\bar{\alpha} < \gamma$ be such that $A \in \cp{\SS \res \xi_{\bar{\alpha}}}^+$. Let $v' : A \to V$ be such that $v'(f) = v(f)(\bar{\alpha})$ if $\bar{\alpha} \in \dom(f)$. By fineness $A \Vdash_\SS [v']_{\dot{\ff}(\SS)} = [\dot{u}_{\bar{\alpha}}]_{\dot{\ff}(\SS)}$ for all $\alpha<\gamma$. Thus $A\Vdash_\SS [\dot{u}_{\bar{\alpha}}]_{\dot{\ff}(\SS)} \in k_{\xi_{\bar{\alpha}}}[M_{\xi_{\bar{\alpha}}}]$, a contradiction.
\end{proof}

The situation described in the previous theorem occurs in several cases, as shown by the following proposition.

\begin{proposition}
	Let $\TT$ be an ideal tower of height $\lambda$ limit ordinal, $\alpha < \lambda$ be an ordinal. Then $\TT \res \alpha$ does not express $\TT$.
\end{proposition}
\begin{proof}
	Suppose by contradiction that there is an $A \in \TT^+$ such that $A \Vdash_\TT M  \supseteq k_\alpha[M_\alpha]$. Then in particular $A \Vdash_\TT [\id_{V_{\alpha+1}}]_{\dot{\ff}(\TT)} \in k_\alpha[M_\alpha]$ hence let $A' \in \TT^+$, $A' \leq A$ be such that $A' \Vdash_\TT [\id_{V_{\alpha+1}}]_{\dot{\ff}(\TT)} = [u]_{\dot{\ff}(\TT)}$ for some $u : O_a \to V$, $a \in \CCC \res \alpha$. Thus by \L o\'s Theorem,
	\[
	B = \bp{f \in O_{V_{\alpha+1}} : ~ f = u(\pi_a(f))} \in \TT^+
	\]
	Since $\vp{B} \leq \vp{\ran(u)} \leq \vp{O_a} \leq \vp{\beth_\alpha} < \vp{V_{\alpha+1}}$, $\dom[B]$ is a non-stationary subset of $V_{\alpha+1}$ by Lemma \ref{lem:small_ns} contradicting Proposition \ref{prop:largeness}.
\end{proof}

\begin{lemma} \label{lem:max_push}
	Let $\SS$ be a $\CCC$-system of filters of length $\lambda$, $\gamma$ be a cardinal such that $\vp{a} \leq \gamma$ for all $a \in \CCC$. Then $j_{\dot{\ff}(\SS)}(\gamma) < \cp{(\beth_\lambda \cdot 2^\gamma)^+}^V$.
\end{lemma}
\begin{proof}
	Consider the set $U$ of functions $u : a \to \gamma$ for some $a \in \CCC$. The total number of such functions is bounded by
	\[
	\vp{\CCC} \cdot \gamma^{\sup_{a \in \CCC} \vp{a}} \leq \beth_\lambda \cdot \gamma^\gamma = \beth_\lambda \cdot 2^\gamma = \delta
	\]
	Let $U = \ap{u_\alpha : ~ \alpha < \delta}$, $\AAA_\alpha$ be the maximal antichain in $\SS^+$ deciding the value of $[u_\alpha]_{\dot{\ff}(\SS)}$, and let $X_\alpha = \bp{\xi : ~ \exists A \in \AAA_\alpha ~ A \Vdash [\check{u}_\alpha]_{\dot{\ff}(\SS)} = \check{\xi}}$, $X = \bigcup_{\alpha < \delta} X_\alpha$. Since $\vp{X_\alpha} \leq \vp{\AAA_\alpha} \leq \beth_\lambda$, we have that $\vp{X} \leq \delta \cdot \beth_\lambda = \delta$.
	Let now $\dot{v}$ be such that $[\dot{v}]_{\dot{\ff}(\SS)} < j_{\dot{\ff}(\SS)}(\gamma)$. Then there is a dense set of $A \in \SS^+$ such that $A \Vdash [\dot{v}]_{\dot{\ff}(\SS)} = [\check{u}_\alpha]_{\dot{\ff}(\SS)} \Rightarrow A \Vdash [\dot{v}]_{\dot{\ff}(\SS)} \in \check{X}$. Thus $j_{\dot{\ff}(\SS)}(\gamma) \subseteq X$ (actually, $j_{\dot{\ff}(\SS)}(\gamma) = X$) and $\vp{X} \leq \delta$, hence $j_{\dot{\ff}(\SS)}(\gamma) < (\delta^+)^V$.
\end{proof}

\begin{proposition}
	Let $\EE$ be a $\ap{\kappa, \lambda}$-ideal extender such that $\lambda = \beth_\lambda$. Suppose that $\1 \Vdash j_{\dot{\ff}(\EE)}(\gamma) \geq \lambda$ for some $\gamma < \lambda$. Then $\EE \res \alpha$ does not express $\EE$ for any $\alpha < \lambda$.
\end{proposition}
\begin{proof}
	Since $j_{\dot{\ff}(\EE)}(\gamma) \geq \lambda$, $\kappa_{\bp{\alpha}} \leq \gamma$ for any $\alpha < \lambda$ hence we can apply Lemma \ref{lem:max_push} to obtain $j_{\dot{\ff}(\EE \res \alpha)}(\gamma) < \cp{(\beth_\alpha \cdot 2^\gamma)^+}^V$ which is smaller than $\lambda$ since $\alpha, \gamma < \lambda$ and $\lambda$ is a $\beth$-fixed point. It follows that the critical point of $k_\alpha : M_\alpha \to M$ is at most $j_{\dot{\ff}(\EE \res \alpha)}(\gamma)$ and in particular $k_\alpha[M_\alpha] \neq M$.
\end{proof}

We remark that the conditions of the previous proposition are often fulfilled. In particular they hold whenever $\EE$ is the $\ap{\kappa, \lambda}$-ideal extender derived from an embedding $j$ and the length $\lambda$ is a $\beth$-fixed point but not a $j$-fixed point.

% % % % % % % % % % % % % % % % % % % % % % % % % % % % % % % % % % % % % % % % % % % % % % % % % % %

\subsection{Distinction between generic large cardinal properties} \label{ssec:distinction_glcp}

Let $j : V \to M \subseteq M[G]$ be a generic elementary embedding with critical point $\kappa$. In this section we provide examples separating the following generic large cardinals notions at a successor cardinal $\kappa = \gamma^+$.
\begin{itemize}
	\item $j$ is almost superstrong if $V_{j(\kappa)}^M \prec V_{\gamma^+}^{V[G]}$;
	\item $j$ is superstrong if it is $j(\kappa)$-strong;
	\item $j$ is almost huge if it is ${<}j(\kappa)$-closed.
\end{itemize}
These examples will all be obtained by collapsing with $\CC = \Coll(\gamma, {<}\kappa)$ a suitable large cardinal embedding in $V$, so that by Theorem \ref{thm:coll_embedding} a generic large cardinal embedding $j^\CC$ is obtained with the desired properties.

\begin{proposition}
	Let $\kappa$ be a $2$-superstrong cardinal. Then there is a generic elementary embedding on $\kappa = \gamma^+$ that is almost superstrong and not superstrong.
\end{proposition}
\begin{proof}
	Let $j$ be a $2$-superstrong embedding with critical point $\kappa$, and let $\EEE$ be the $\ap{\kappa,j(\kappa)}$-extender derived from $j$. Since $V$ models that $\EEE$ is a superstrong $\ap{\kappa, j(\kappa)}$-extender, by elementarity $M$ models that $j(\EEE)$ is a superstrong $\ap{j(\kappa), j^2(\kappa)}$-extender. Thus
	\[
	\begin{array}{rrcl}
		j_{j(\EEE)}: & M & \longrightarrow & N = \Ult(M, j(\EEE)) \supseteq M_{j^2(\kappa)} = V_{j^2(\kappa)} \\
		   & j(\kappa) & \longmapsto & j^2(\kappa)
	\end{array}
	\]
	Since $M \subseteq V$, also $\Ult(V, j(\EEE)) \supseteq \Ult(M, j(\EEE)) \supseteq V_{j^2(\kappa)}$ hence $j(\kappa)$ is superstrong as witnessed by $j(\EEE)$ also in $V$.

	Consider now $\CC = \Coll(\gamma, {<}\kappa)$, $j_1$ induced by $\CC$ and $j_\EEE$, $j_2$ induced by $j(\CC)$ and $j_{j(\EEE)}$. By Proposition \ref{prop:coll_closure_strength}, $j_1$ and $j_2$ are still superstrong and we get the following diagram, where all the inclusions are superstrong:
	\begin{center}
		\begin{tikzpicture}[xscale=2,yscale=1]
				\node (J0) at (-0.5, 0) {$j_0 :$};
				\node (J1) at (0.5, -1) {$j_1 :$};
				\node (A0_0) at (0, 0) {$V^\CC$};
				\node (A1_0) at (1, 0) {$M^{j(\CC)}$};
				\node (A1_1) at (1, -1) {$V^{j(\CC)}$};
				\node (A2_1) at (2, -1) {$N^{j^2(\CC)}$};
				\node (A2_2) at (2, -2) {$V^{j^2(\CC)}$};
				\path (A0_0) edge [->] node [auto] {} (A1_0);
				\path (A1_1) edge [->] node [auto,swap] {} (A2_1);
				\path (A1_0) edge[draw=none] node [sloped] {$\subseteq$} (A1_1);
				\path (A2_1) edge[draw=none] node [sloped] {$\subseteq$} (A2_2);
		\end{tikzpicture}
	\end{center}
	Thus $j_0$ considered as a generic elementary embedding in $V^{j^2(\CC)}$ is almost superstrong:
	\[
	M_{j(\kappa)}^{j(\CC)} = V_{j(\kappa)}^{j(\CC)} \prec N_{j^2(\kappa)}^{j^2(\CC)} = V_{j^2(\kappa)}^{j^2(\CC)} = V_{\gamma^+}^{j^2(\CC)}
	\]
	but not superstrong, since $V_{\gamma+1}^{j(\CC)}$ has cardinality $\delta \in (\gamma, j^2(\kappa))$ in $V^{j(\CC)}$ hence in $V^{j^2(\CC)}$ its cardinality is collapsed to $\gamma$ and bijection between $\gamma$ and $V_{\gamma+1}^{j(\CC)}$ is added.
\end{proof}

\begin{proposition}
	Let $\kappa$ be a $2$-huge cardinal. Then there is a generic elementary embedding on $\kappa = \gamma^+$ that is superstrong and not ${<}\omega$-closed.
\end{proposition}
\begin{proof}
	Let $j$ be a $2$-huge embedding in $V$ with critical point $\kappa$. Then we can derive a $\ap{\kappa, j(\kappa)+\omega}$-tower $\TT$ from $j$, so that $j_\TT$ is still $\kappa+\omega$-superstrong but by Theorem \ref{thm:not_closure} is not closed under $\omega$-sequences.
	
	Let $j^\CC : V^\CC \to M^{j(\CC)}$ be derived from $j_\TT$ as in Theorem \ref{thm:coll_embedding}. Since $j_\TT$ is $\kappa+\omega$-superstrong, by Proposition \ref{prop:coll_closure_strength} $j^\CC$ is still $\kappa+\omega$-superstrong (hence superstrong). Moreover $j^\CC$ is not ${<}\omega$-closed. In fact given any $A \in {}^\omega M \setminus M$, $j(\CC)$ cannot add $A$ since it is a set of size $\omega$ and $\CC$ is closed under $\omega$-sequences.
\end{proof}

	\section{Conclusions and open problems} \label{sec:conclusions}
		% !TEX root = GenericLargeCardinals.tex

In the last section we investigated some topics related to the definability of generic large cardinal properties. We gave a unified treatment of extenders and towers, and some partial results on how generic large cardinal embeddings are induced by set-sized objects. However, many questions remain open. We list them according to the ordering of sections of this paper.

\begin{question}
	Is $(j_\SSS \res a)^{-1}$ the unique element of $\bigcap j_\SSS[F_a]$?
\end{question}

\begin{question}
	Assume that $j$ is a ${<}\gamma$-closed embedding in $V[G]$, with $G$ $V$-generic for $\BB$. Can this be witnessed by a generic $\gamma$-extender of sufficient length \emph{independently of the chain condition satisfied by $\BB$}?
\end{question}

In Theorem \ref{thm:e_dominates_s}, we showed that a generic extender can have enough expressive power to approximate any other generic $\CCC$-system of ultrafilters. This observation suggests the following question.

\begin{question}
	Assume $\kappa$ has ideally property $P$, can this be witnessed by an extender in the ground model $V$?
\end{question}

Since $\kappa$ has ideally property $P$, there is a $\CCC$-system of filters $\SS$ in $V$ which witness property $P$ for $\kappa$. However, in general $\SS$ might not be an extender (e.g. a tower or a $\gamma$-extender). We already know that a generic extender $\EEE$ is able to fully approximate $\SS$, is this possible also for an ideal extender $\EE$ in $V$? We believe that this question could be an important cornerstone to uncover the following.

\begin{question}
	Is having ideally property $P$ equivalent to having generically property $P$?
\end{question}

In Section \ref{sec:generic_lc} we characterized large cardinal properties such as strongness, superstrongness and hugeness of an ultrapower embedding in terms of combinatoric properties of a $\CCC$-system of filters. However, there is a growing set of results on very large cardinals (see among others \cite{dimonte:I0GCH,dimonte:I0,woodin:beyondI0}), for which the notions of ${<}\gamma$-guessing and ${<}\gamma$-splitting are not enough.

\begin{question}
	How can large cardinal properties beyond hugeness be characterized for $\CCC$-systems of filters? 
\end{question}

	\bibliographystyle{amsplain}
	\bibliography{Bibliography}
\end{document}